\documentclass[12 pt]{amsart}
\usepackage[margin=1in]{geometry}
\usepackage{amssymb, amsmath, amsthm}
\usepackage{hyperref}
\usepackage{fullpage}
\usepackage{wasysym}
\usepackage{setspace}
\usepackage{cancel}
\usepackage[all]{xy}
\usepackage{empheq}
\usepackage{extpfeil}

\newtheorem{theorem}{Theorem}[section]
\newtheorem{lemma}[theorem]{Lemma}
\newtheorem{proposition}[theorem]{Proposition}
\newtheorem{corollary}[theorem]{Corollary}

\newtheorem{THM}{Theorem}

\theoremstyle{definition}
\newtheorem{remark}[theorem]{Remark}
\newtheorem{definition}[theorem]{Definition}

\def\beq{\begin{eqnarray*}}
\def\eeq{\end{eqnarray*}}

\def\R{\mathbb{R}}
\def\Z{\mathbb{Z}}

\def\N{\mathbb{N}}

\def\DD{\mathcal{D}}
\def\incl{\hookrightarrow}
\def\to{\rightarrow}
\def\tH{\widetilde{H}}

\def\dim{\mathrm{dim}\>}

\def\x{\times}
\def\d{\partial}
\def\phi{\varphi}

\def\Emb{\mathrm{Emb}}

\def\K{\mathcal{K}}
\def\L{\mathcal{L}}
\def\tL{\overset{\circ}{\mathcal{L}}}
\def\V{\mathcal{V}}
\def\W{\mathcal{W}}
\def\CD{\mathcal{CD}}
\def\TD{\mathcal{TD}}

\def\Aut{\mathrm{Aut}}

\def\im{\mathrm{im}}

\def\NN{\langle N \rangle}

\def\SS{\mathfrak{S}}

\def\co{\colon\thinspace}

    \title{The Milnor triple linking number of string links by cut-and-paste topology}
    \author{Robin Koytcheff}
    \email{rmjk@uvic.ca}
    \thanks{This work was supported by NSF grant DMS--1004610.}
    \address{Mathematics \& Statistics, University of Victoria, 
    Victoria, BC, Canada}

\begin{document}

\begin{abstract}
Bott and Taubes constructed knot invariants by integrating differential forms along the fiber of a bundle over the space of knots, generalizing the Gauss linking integral.  Their techniques were later used to construct real cohomology classes in 
spaces of knots and links in higher-dimensional Euclidean spaces.  In previous work, we constructed cohomology classes in knot spaces with arbitrary coefficients by integrating via a Pontrjagin--Thom construction.  We carry out a similar construction over the space of string links, but with a refinement in which configuration spaces are glued together according to the combinatorics of weight systems.
This gluing is somewhat similar to work of Kuperberg and Thurston.  We use a formula of Mellor for weight systems of Milnor invariants, and we thus recover the Milnor triple linking number for string links, which is in some sense the simplest interesting example of a class obtained by this gluing refinement of our previous methods.  Along the way, we find a description of this triple linking number as a degree of a map from the 6--sphere to a quotient of the product of three 2--spheres.
\end{abstract}

\maketitle

\section{Introduction}
Let $\L_k$ be the space of $k$-component string links, ie, the space of embeddings $f\co \coprod_k \R \incl \R^3$ which agree with a fixed linear embedding outside of $\coprod_k [-1,1] \subset \coprod_k \R$.  This paper concerns Milnor's triple linking number \cite{MilnorLinkGroups, MilnorIsotopy} for string links, which is a function from isotopy classes of string links to the integers.  Since $H_0(\L_k)$ is generated by isotopy classes of string links, we can view this link invariant as a cohomology class in $H^0(\L_3)$.  

We build on our previous work \cite{Rbo}.  This work was inspired by the configuration space integrals of Bott and Taubes \cite{Bott-Taubes} and subsequent work based on their methods, which produced real cohomology classes in spaces of knots.  In our work, we replaced integration of differential forms by a Pontrjagin--Thom construction.  This produced cohomology classes with \emph{arbitrary} coefficients.  It is not too difficult to generalize configuration space integrals or the homotopy-theoretic construction of our previous work to spaces of links and string links.  Perhaps less obvious is a refinement of this construction in which configuration spaces are glued along their boundaries.  Here we show that via this refinement, we recover the Milnor triple linking number for string links.  This invariant is already known to be integer-valued, but generalizations of this work should have more novel implications.  

We conjecture that this gluing refinement of our ``homotopy-theoretic Bott--Taubes integrals" produces integral multiples of all the Bott--Taubes/Vassiliev-type cohomology classes of Cattaneo, Cotta-Ramusino and Longoni \cite{Cattaneo}, showing that these classes are rational.  
This work is currently in progress.
In this paper, we will focus only on the specific but interesting example of the Milnor triple linking number, rather than all finite-type invariants or all the cohomology classes of Cattaneo et al.  

As our title suggests, this gluing is inspired by and similar to a construction of Kuperberg and Thurston  \cite{KT} (see also the paper of Lescop \cite{LescopKT}).  This gluing idea was also present in the work of Bott and Taubes \cite[Equation 1.18]{Bott-Taubes}, as once pointed out to the author by Habegger.  The difference between our construction and that of Kuperberg and Thurston is that we do a gluing which is specific to the weight system for the invariant.  Thus, in our work in progress, we generalize this to arbitrary cohomology classes by constructing one glued-up space for each weight system (or graph cocycle), rather than a glued-up space that accounts for all weight systems.  This slightly different approach is necessary to produce bundles whose fibers are nice enough to admit neat embeddings and Pontrjagin--Thom constructions.

\subsection{The Gauss linking integral}
\label{GaussIntegral}
Before stating our main results, we describe the analogue of our constructions in the simpler case of the pairwise linking number.  We first discuss in detail this pairwise linking number in terms of a Pontrjagin--Thom construction and the space of links in the case of \emph{closed} links, where this invariant is completely straightforward.  

For any space $X$, denote the configuration space of $q$ points in $X$ by 
\[C_q(X):=\{(x_1,...,x_q)\in X^q | x_i \neq x_j \forall i \neq j\}.\]
An inclusion $\xymatrix{ X \ar@{^(->}[r]^f & Y}$ induces an inclusion of configuration spaces $\xymatrix{C_q(X) \ar[r]^f & C_q(Y).}$

The linking number of $L \co S^1 \sqcup S^1 \to \R^3$ is then given by the degree of the composition
\begin{equation}
\label{GaussMap}
\xymatrix{
S^1 \x S^1 \ar@{^(->}[r] & C_2(S^1 \sqcup S^1) \ar[r]^-L & C_2(\R^3) \ar[r]^-{\phi_{12}} & S^2} 
\end{equation}
\[
\xymatrix{ & & &\>\>\>\>\>\>\>\>\>\>\>\>\>\>\>\>\>\>\>\>\>\>\>\>\>\>\>\>\>\>\>\>  (x_1,x_2) \ar@{|->}[r] & \frac{x_2 - x_1}{|x_2 - x_1|}}
\]
where $S^1\x S^1$ is one of the two (isomorphic) components of $C_2(S^1 \x S^1)$ where the two points are on distinct circles.  Here the 
rightmost 
map $\phi_{12}$ happens to be a homotopy equivalence.  One way of realizing this degree is to pull back a volume form on $S^2$ to a 2--form $\theta$ and integrate $\theta$ over $S^1 \x S^1$.

Another way of realizing this integer would be to start by embedding $S^1 \x S^1$ into some Euclidean space $\R^N$.  The Pontrjagin--Thom collapse map then gives a map from $S^N$ to the Thom space $(S^1 \x S^1)^\nu$ of the normal bundle $\nu$ of the embedding:
\[
\xymatrix{ S^N \ar[r]^-\tau & (S^1 \x S^1)^\nu}
\]
The normal bundle $\nu$ has fiber dimension $N-2$, so using the Thom isomorphism we have in integral cohomology
\[
\xymatrix{
H^2(S^1 \x S^1) \ar[r]^-{\mbox{Th }\cong} & H^N((S^1 \x S^1)^\nu) \ar[r]^-{\tau^*} & H^N(S^N) \cong \Z.
}
\]
If we again let $\theta$ denote the pullback to $S^1 \x S^1$ of a volume form on $S^2$ via (\ref{GaussMap}), the image of the cohomology class $[\theta]$ under the above composition is the linking number. 

To describe this as a function on the space of links, let $\overset{\circ}{\L}_k$ denote the space $\Emb(\coprod_k S^1, \R^3)$ of $k$-component closed links (as opposed to string links).  We can rewrite (\ref{GaussMap}) as
\begin{equation}
\label{GaussMapLinkSpace}
\xymatrix{
\overset{\circ}{\L}_2 \x S^1 \x S^1 \ar@{^(->}[r] & \overset{\circ}{\L}_2 \x C_2(S^1 \sqcup S^1) \ar[r] & C_2(\R^3) \ar[r] & S^2. 
}
\end{equation}
Then we can pull back a volume form on $S^2$ (or generator of $H^2(S^2)$) to a 2--form $\theta$ on $\tL_2 \x S^1 \x S^1$ (or a cohomology class $[\theta]$).  The latter space is a trivial $(S^1\x S^1)$-bundle over $\tL_2$, and integrating $\theta$ along the fiber gives a function on $\tL_2$, which is the linking number.  

This integration along the fiber can also be described via a Pontrjagin--Thom construction as follows.  We embed the bundle into a trivial $\R^N$-bundle
\begin{equation}
\label{EmbeddingGaussLinkingNumber}
\tL_2 \x S^1 \x S^1 \incl \tL_2 \x \R^N
\end{equation}
and collapse a complement of the tubular neighborhood of the embedding to get a map
\begin{equation}
\label{ThomCollapseGaussLinkingNumber}
\Sigma^N (\tL_2)_+ \to (\tL_2 \x S^1 \x S^1)^\nu
\end{equation}
where the left-hand side is the suspension of $\tL_2$ with a disjoint basepoint, and where the right-hand side is the Thom space of the normal bundle $\nu$ of the embedding.  Using the Thom isomorphism and suspension isomorphism, we have in cohomology
\[
\xymatrix{
H^2(\tL_2 \x S^1 \x S^1) \ar[r]^-{\mbox{Th }\cong} & H^N((\tL_2 \x S^1 \x S^1)^\nu) \ar[r] & H^N(\Sigma^N \tL_2) \ar[r]^-{\mbox{susp }\cong} & H^0(\L_2)
}
\]
The image of $[\theta]$ under this composite is the linking number.  Those familiar with 
spectra in 
stable homotopy theory will realize that the maps (\ref{ThomCollapseGaussLinkingNumber}) for all $N$ can be 
more concisely 
rewritten as a map from a suspension spectrum to a Thom spectrum.  This was the perspective taken in 
our previous work
 \cite{Rbo}. 

\subsection{Gauss linking integral for string links}
\label{GaussIntegralStringLinks}
In this paper we consider string links because the Milnor triple linking number is a well defined integer only for string links.  (For closed links, it is only defined modulo the gcd of the pairwise linking numbers.)

\begin{definition}
\label{StringLinkDef}
Fix a positive real number $R$.  We define a \emph{string link with $k$ components} as a smooth embedding $L\co \coprod_k \R \incl \R^3$ such that 
\begin{itemize}
\item
$L$ takes $\coprod_k [-R,R]$ into $B_R \subset \R^3$, a ball of radius $R$ around the origin 
\item
outside of $\coprod_k [-R,R]$, $L$ agrees with $2k$ fixed linear maps $f_i^-\co (-\infty, -R] \to \R^3$, $f_i^+\co [R,\infty) \to \R^3$, $i=1,...,k$  
\item
the directions ${v}^{\pm}_1,...,{v}_k^{\pm}$ of the fixed linear maps $f_1^\pm, ...,f_k^\pm$ are distinct vectors in $S^2$.  
\end{itemize}
For definiteness, we choose ${v}^{\pm}_i$ to be the unit vector in the direction $(\frac{k+1}{2} - i, \pm 1, 0)$.  Also note that the orientation of each copy of $\R$ gives an orientation of a string link.
 \end{definition}



\begin{figure}[h]
\begin{center}
\includegraphics[height=6pc]{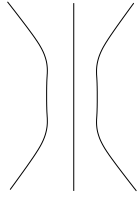}
\end{center}
\caption{An unlink in the setting of (three-component) string links.}
\end{figure}



Now to treat the simplest example in the string link setting, 
consider the linking number of a two-component string link.  
We start by replacing 
(\ref{GaussMap}) by 
\begin{equation}
\label{GaussMapStringLink}
\xymatrix{
 \R \x \R \ar@{^(->}[r] & C_2(\R \sqcup \R) \ar[r] & C_2(\R^3) \ar[r]^-{\phi_{12}} & S^2. 
}
\end{equation}
If we compactified $\R\x \R$ to a square disk, the map to $S^2$ would be ill defined at (some of) the corners, taking different constant values on adjacent edges.  
However, we can keep track of the relative rates of approach to infinity of the two points.  Such a compactification yields an octagonal disk $\octagon$; it is just the blowup of the square at the four corners.  
Thus we have
\begin{equation}
\label{GaussMapStringLinkCptfd}
\xymatrix{\octagon \ar@{^(->}[r] & C_2[\R^3] \ar[r]^-{\phi_{12}} & S^2} 
\end{equation}
where $C_2[\R^3]$ is the Axelrod--Singer compactification of $C_2(\R^3)$.  (The reader unfamiliar with this intermediate space may ignore it for now and wait until it is reviewed in Section \ref{Background}).

The map $\phi_{12}\co \octagon \to S^2$ takes the boundary of $\octagon$ into an arc $A$ on a great circle in $S^2$. 
Namely, 
by our choices of ${v}_i^\pm$, $A$ is such an arc contained in the set $\{z=0, x\geq 0\}$.  We can find a 2--form $\alpha$ on $S^2$ which vanishes on $A$ but is cohomologous to a unit volume form.  We can then pull it back to a 2--form $\theta:=\phi_{12}^*\alpha$ on $\octagon$ and integrate $\theta$ over $\octagon$.\footnote{Since $\theta$ is 0 on the boundary of $\octagon$, one can show (using Stokes' Theorem for integration along the fiber of the bundle $\L_2 \x \octagon \to \L_2$) that this integral is an isotopy invariant.  Cf. equation (\ref{Stokes}) in Section \ref{ConfigSpaceIntegrals}.}  More topologically, we have maps 
\[
\xymatrix{
S^2 \cong \octagon / \d \octagon \ar[r]^-{\phi_{12}} & S^2 / A  
 & S^2 \ar@{->>}[l]_-{\mathrm{quotient}}
}
\]
where $[\alpha] \in H^2(S^2/A;\Z)$ is a class that maps to a generator of $H^2(S^2;\Z)$.  The pairing of $[\theta]$ with the fundamental class $[\octagon, \d \octagon]$ gives 
a number $\ell_\alpha$, which is the pairwise linking number (at least up to a sign, depending on conventions and choices of orientations).  
We can take $\alpha$ to be supported in a (small) disk in $S^2$ and quotient by the complement of this support:
\[
\xymatrix{S^2/A 
\ar[r]^-{q(\alpha)} & S^2}
\]
Then one can check that the degree of the composition $q(\alpha)\circ \phi_{12}$ is our linking number $\ell_\alpha$.
In this case any choice of such an $\alpha$ will yield the same $\ell_\alpha$ (since $S^2 \setminus A$ is connected), though we will see that for the triple linking number, the analogous choice \emph{does} matter. 

Moving towards a Pontrjagin--Thom map with the space of string links, we parametrize the map (\ref{GaussMapStringLinkCptfd}) as in (\ref{GaussMapLinkSpace}) and pull back $[\alpha]$ as above to a class $[\theta]$ on $\L_2 \x \octagon$.  We can embed 
\[
\xymatrix{\L_2 \x \octagon \ar@{^(->}[r] & \L_2 \x \R^M \x[0,\infty)^N}
\]
in a way that preserves the corner structure of the octagonal disk $\octagon$.  
This embedding has a  tubular neighborhood  diffeomorphic to the normal bundle $\nu$.  Collapsing not only the complement of the tubular neighborhood, but also the boundary of $\R^M \x[0,\infty)^N$, gives a collapse map to a quotient of Thom spaces:
\[
\xymatrix{\Sigma^{M+N} \L_2 \ar[r] & (\L_2 \x \octagon)^\nu / (\L_2 \x \d \octagon)^\nu}
\]
The induced map in cohomology, together with the relative Thom isomorphism and suspension isomorphism, takes the class $[\theta]$ to the pairwise linking number.
 \\

\subsection{Main Results}

One main motivation for this paper was the question of exactly which knot and link invariants can be produced via the Pontrjagin--Thom construction of our previous work \cite{Rbo}, or a modification of it.  In this paper we begin to answer that question.  We glue four different configuration space bundles over the space $\L_3$ of three-component string links to obtain one bundle $F_g \to E_g\to \L_3$ whose fiber is a (6--dimensional) gluing of the four configuration space fibers.  Along the way to proving a statement about our Pontrjagin--Thom construction, we are led to the following theorem.

\begin{THM}
\label{DegreeThm}
The Milnor triple linking number for string links can be expressed as the pairing of a 6--dimensional cohomology class $[\beta]$ with the fundamental class of $S^6$.  This class $[\beta]$ is defined via maps
\[
\xymatrix{ S^6 \ar[r]^-\Phi & S^2 \x S^2 \x S^2 / \DD & S^2 \x S^2 \x S^2 \ar@{->>}[l]}
\]
where $\DD$ is a union of positive-codimension subsets of $S^2\x S^2 \x S^2$, where the right-hand map is the quotient, and where the left-hand map is induced by the link.  Specifically, $[\beta]$ is the pullback via the left-hand map of a certain class $[\alpha]\in H^6(S^2\x S^2 \x S^2 / \DD; \Z)$ which maps to a generator of $H^6(S^2 \x S^2 \x S^2; \Z)$ via the right-hand map.
\end{THM}

In this Theorem, $S^6$ is related to the glued fiber $F_g$ mentioned above.  This glued fiber modulo its boundary is not exactly $S^6$, but there is a quotient map $S^6 \to F_g/\d F_g$ which induces an isomorphism in top (co)homology.  The map $\Phi$ is induced by a map $F_g \to S^2 \x S^2 \x S^2$ which is a generalization of the Gauss map $\phi_{12}$ for links.  (It sends a configuration to unit vector differences between three pairs of points in the configuration.)  The codimension--1 subspace $\DD$ is thus the image of $\d F_g$.  
This boundary $\d F_g$ of the glued fiber consists entirely of configurations where points approach infinity.

We will see that $\DD$ is contained in a larger set $\widetilde{\DD}$ (see Proposition \ref{DegenerateLocus}) which is the union of a 6--ball and a codimension--1 set.  This set $\widetilde{\DD}$ separates $S^2 \x S^2 \x S^2$ into two components.
We will see that we need to choose $[\alpha]$ to have support in a certain one of these two components (see Proposition \ref{OnTheNose}).  However, the support of $[\alpha]$ may be taken to be arbitrarily small.
Thus in the Theorem above, we could replace $\DD$ by $\widetilde{\DD}$ or even by the complement of an open set contained in one of these two components.
Take the quotient $q(\alpha)\co S^2\x S^2 \x S^2 \to S(\alpha)$ by such a complement of an open set.  The resulting quotient space $S(\alpha)$ is then a 6-sphere.  So we can express the triple linking number as the \emph{degree} of $q(\alpha) \circ \Phi$.

Returning to our Pontrjagin--Thom construction, we show (in Lemma \ref{GluedEmbedsInTrivial}) that we can embed the bundle $E_g \to \L_3$ into a trivial fiber bundle  
\begin{equation}
\label{EmbedInTrivialBundle}
\xymatrix{
E_g  \ar@{^(->}[r] & B_{M,N}(R) \x \L_3
}
\end{equation}
in a way that preserves the corner structure.  The fiber of this trivial bundle, called $B_{M,N}(R)$, is a codimension--0 subset of a ball of radius $R$ in $\R^{M+N}$, where the number $N$ is related to the codimension of the corners in $E_g$.  The space $B_{M,N}(R)$ is homeomorphic to a ball, but it is not quite a manifold with corners.  Nonetheless, this embedding has a tubular neighborhood $\eta(E_g)$ diffeomorphic to the normal bundle, which is enough for a Pontrjagin--Thom map 
\[
\tau\co (B_{M,N}(R) \x \L_3, \d(B_{M,N}(R)) \x \L_3) \to (B_{M,N}(R) \x \L_3, \> (B_{M,N}(R) \x \L_3 - \eta(E_g)) \cup \d E_g ).
\]
Below, we rewrite this as a map of quotients rather than pairs.  We let $E_g^{\nu_{M,N}}$ denote the Thom space of the normal bundle $\nu_{M,N}\to E_g$ to the embedding (\ref{EmbedInTrivialBundle}).
\begin{THM}
\label{PTThm}
Let $[\beta]$ be a cohomology class as in Theorem 1 above.  The Pontrjagin--Thom collapse map 
\[
\tau\co \Sigma^{M+N} \L_3 \to E_g^{\nu_{M,N}} / \d E_g^{\nu_{M,N}}
\]
induces a map $\tau^*$ in cohomology such that $\tau^* [\beta] \in H^0(\L_3)$ is the Milnor triple linking number for string links.  The map is independent of the choice of $M$ or $N$, provided both are sufficiently large.  Equivalently, in the language of stable homotopy theory, the collapse map 
\[
\tau\co \Sigma^{\infty+6} \L_3 \to E_g^\nu / \d E_g ^\nu
\]
from the suspension spectrum of $\L_3$ to the Thom spectrum of the stable normal bundle $\nu \to E_g$ (modulo its boundary) induces a map $\tau^*$ in cohomology such that $\tau^*[\beta]$ is the triple linking number for string links.  (The ``$\infty +6$" superscript reflects the fact that this map sends the $n^\mathrm{th}$ space of the suspension spectrum to the Thom space of the normal bundle whose fiber dimension is $n-6$, where 6 is the fiber dimension of the bundle $E_g\to \L_3$.)
\end{THM}


\subsection{Organization of the paper}
The paper is organized as follows.  Section \ref{Background} contains a 
review of finite-type invariants, configuration space integrals, and Milnor invariants, especially for string links.
Experts may wish to skip this Section and refer back to it as necessary.

Section \ref{MilnorIntegrals} describes the triple linking number as a sum of configuration space integrals, using Mellor's formula for weight systems of Milnor invariants as well as results on configuration space integrals for homotopy string link invariants.  
We examine some known facts about how configuration space integrals combine to make an invariant, specialized to the example of the triple linking number;
this is crucial for determining the gluing construction.  

Section \ref{GluingMfds} starts with some general facts about manifolds with faces, which are a nice kind of manifolds with corners.  We then glue the appropriate bundles over the link space $\L_3$, and show that the glued space is a manifold with faces.

Section \ref{ProofThm1} proves Theorem 1, expressing the triple linking number as a degree.

Section \ref{IntegrationViaPT}  shows that integration along the fiber coincides with a certain Pontrjagin--Thom construction.  We have left this arguably foundational material to one of the last Sections because it is not needed for Theorem 1; in addition, some of the ideas of Section \ref{GluingMfds} are useful for making sense of Pontrjagin--Thom constructions for manifolds with corners.

Finally, Section \ref{MilnorViaPT} proves Theorem 2, expressing the triple linking number via a Pontrjagin--Thom construction.


\subsection{Acknowledgments}  The author  was supported by NSF grant DMS--1004610. 
He thanks Tom Goodwillie for illuminating conversations, suggestions, and corrections, and for listening to his ideas.  He thanks Ismar Voli\'{c} and Brian Munson for conversations and work
 that motivated this paper.  Conversations with Chris Kottke about blowups and a comment of the referee of our paper \cite{MunsonVolicHtpyLinks} were useful in resolving an issue related to defining configuration space integrals for string links.  The author thanks Slava Krushkal for a helpful conversation on Milnor invariants.  He thanks Dinakar Muthiah, Sam Molcho, and especially Danny Gillam for their interest in learning about configuration space integrals, which led to useful conversations.  He thanks the referee for many useful comments and suggestions.

\section{Background material}
\label{Background}

This Section reviews background material, including finite-type invariants, configuration space integrals, and Milnor invariants, all in the setting of string links.  The most subtle points are extending the definition of configuration space integrals for string links (section \ref{TechnicalForStringLinks}) and two different ways of orienting diagrams (much of Section \ref{FTTrivalent}).  The former point is also addressed in our recent work with Munson and Voli\'{c} \cite{MunsonVolicHtpyLinks}, while the latter is addressed in the AB Thesis of D Thurston \cite{Thurston} and his paper with Kuperberg \cite{KT}.

\subsection{Finite-type invariants}

Here we give a brief review of Vassiliev invariants (ie, finite-type invariants) of string links.  For an introduction to the subject, we recommend the paper of Bar-Natan \cite{BarNatan} or the book of Prasolov and Sossinsky \cite{PrasolovSossinsky}.  
We will often write ``link" to mean ``string link."  

We will consider finite-type invariants over the field $\R$ (or arguably over the ring $\Z$, in the case of the triple linking number).
We write $\V^k_m$ for the $\R$-vector space (or $\Z$--module) of type-$m$ invariants of $k$-component links.
%
%
For string links, one considers (\emph{degree} $m$) \emph{chord diagrams} on $k$ strands, as in Figure \ref{AChordDiagram}.  We let $\CD_m^k$ be the vector space of such diagrams.


\begin{figure}[h]
\begin{center}
\includegraphics[height=3pc]{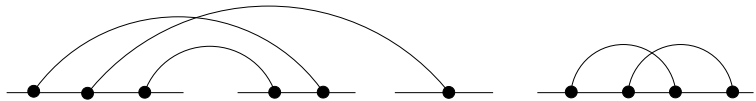}
\end{center}
\caption{A chord diagram in $\CD_5^4$.}
\label{AChordDiagram}
\end{figure}


As in the finite-type theory of knots, there is a canonical injection
\begin{equation}
\label{CanonWeightSystem}
\xymatrix{ w\co \V^k_m / \V^k_{m-1} \ar@{^(->}[r] & (\CD_m^k)^*}.
\end{equation}
Furthermore, any element $W$ in the image of $w$ satisfies the \emph{one-term relation} (1T) and the \emph{four-term relation} (4T).  (See Bar-Natan's paper \cite{BNHtpyLinks} for these relations in the string link setting.)


Call an element in  $(\CD_m^k)^*$ satisfying these relations a 
\emph{weight system}, and let $\W_m^k$ be the subspace of $(\CD_m^k)^*$ of all weight systems.  
The main theorem in the theory of finite-type invariants is that the linear injection $w$ in (\ref{CanonWeightSystem}) is in fact an isomorphism of $\R$-vector spaces:
\begin{equation}
\label{FiniteTypeIso}
\xymatrix{
\V^k_m / \V^k_{m-1} \ar[r]_-w & \W^k_m \ar@/_1.5pc/[l]
}
\end{equation}

\subsection{Finite-type invariants and trivalent diagrams}
\label{FTTrivalent}
One way of constructing the inverse map $\W^k_m \to \V^k_m/\V^k_{m-1}$ 
is via configuration space integrals.  This approach requires enlarging
 $\CD_m^k$ to 
a space of trivalent diagrams $\TD_m^k$.  
(See \cite{MunsonVolicHtpyLinks} for a thorough treatment of this space.)
We note some important conventions below.

\begin{figure}[h]
\begin{center}
\includegraphics[height=5pc]{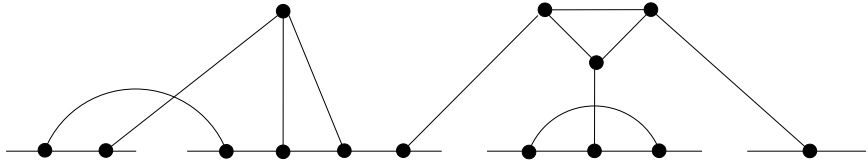} \\
\end{center}
\caption{A trivalent diagram in $\TD_7^4$.  The horizontal lines are the strands.  The Lie orientation is determined by this drawing of the diagram in the plane.}
\end{figure}

Vertices on the strands are called \emph{interval vertices}, and the remaining vertices are called \emph{free vertices}.  
We call an edge between two interval vertices a \emph{chord}.  These diagrams are equipped with a \emph{Lie orientation}.  That means that at every vertex we have a cyclic order of the three emanating edge-ends (say, determined by a planar embedding), and in $\TD_m^k$, the effect of reversing the cyclic order at one vertex is to multiply the diagram by $-1$.  We put a canonical cyclic order at each interval vertex by drawing all chords \emph{above} the strands, so there is a well defined inclusion $\CD_m^k \subset \TD_m^k$.  This Lie orientation is equivalent to one determined by an ordering of the vertices and an orientation on every edge, where changing the vertex-ordering by an odd permutation multiplies the diagram by $-1$, as does reversing the orientation on one edge.  (See the work of D Thurston \cite[Appendix B]{Thurston} or his work with Kuperberg \cite[Section 3.1]{KT}.)

If we impose the STU relation below, then every trivalent diagram can be expressed as a sum of chord diagrams.  
\begin{equation}
\label{UnlabeledSTU}
\raisebox{-3pc}{\includegraphics[height=6pc]{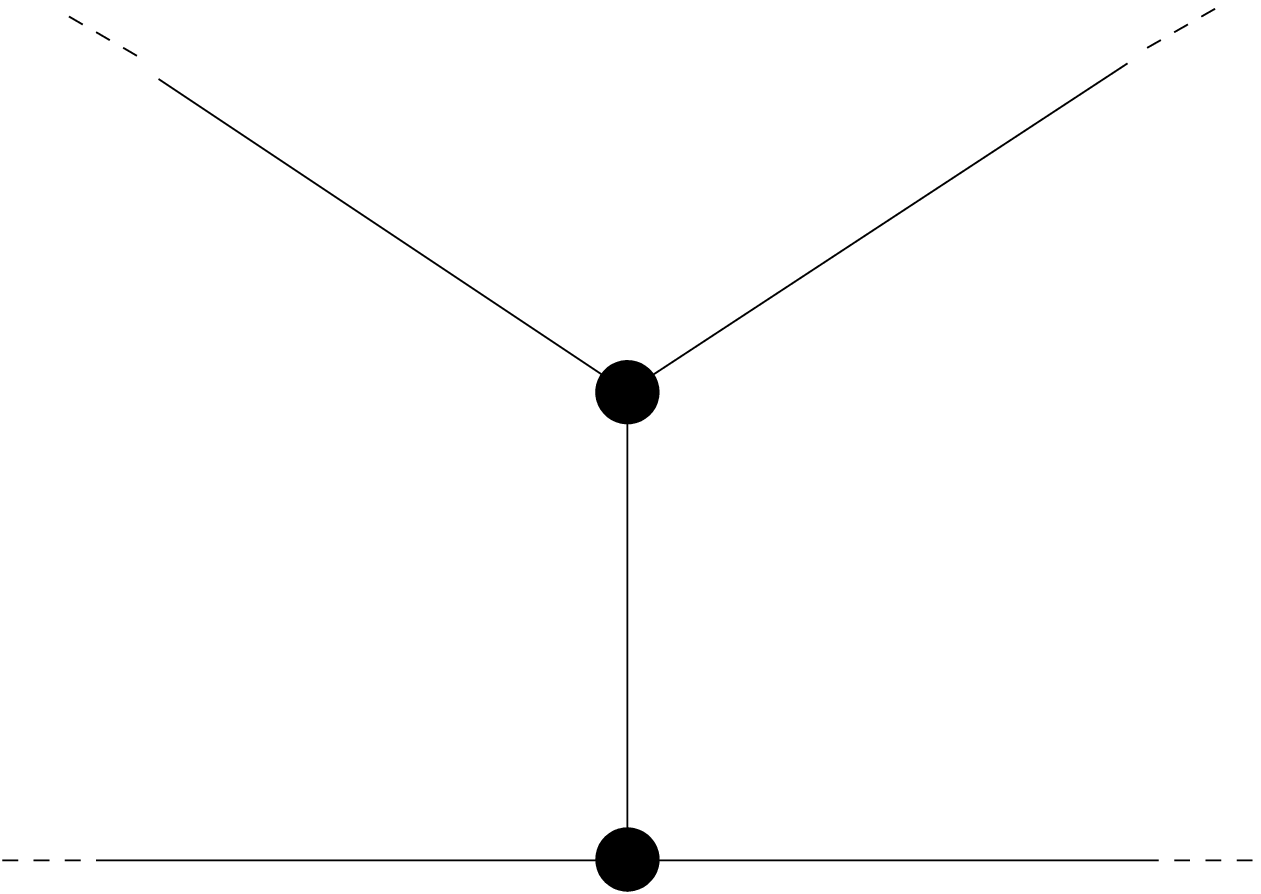}} =
\raisebox{-3pc}{\includegraphics[height=6pc]{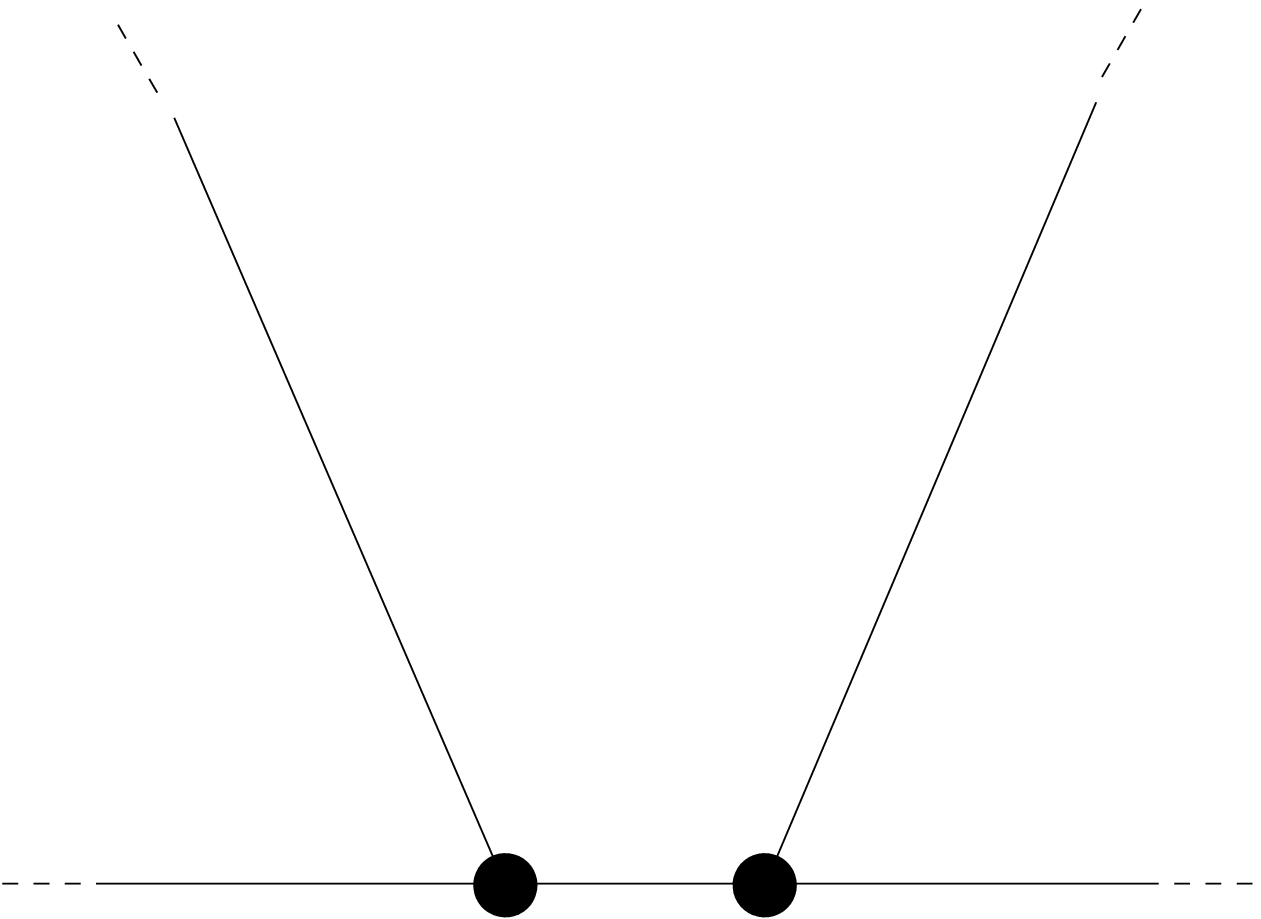}} -
\raisebox{-3pc}{\includegraphics[height=6pc]{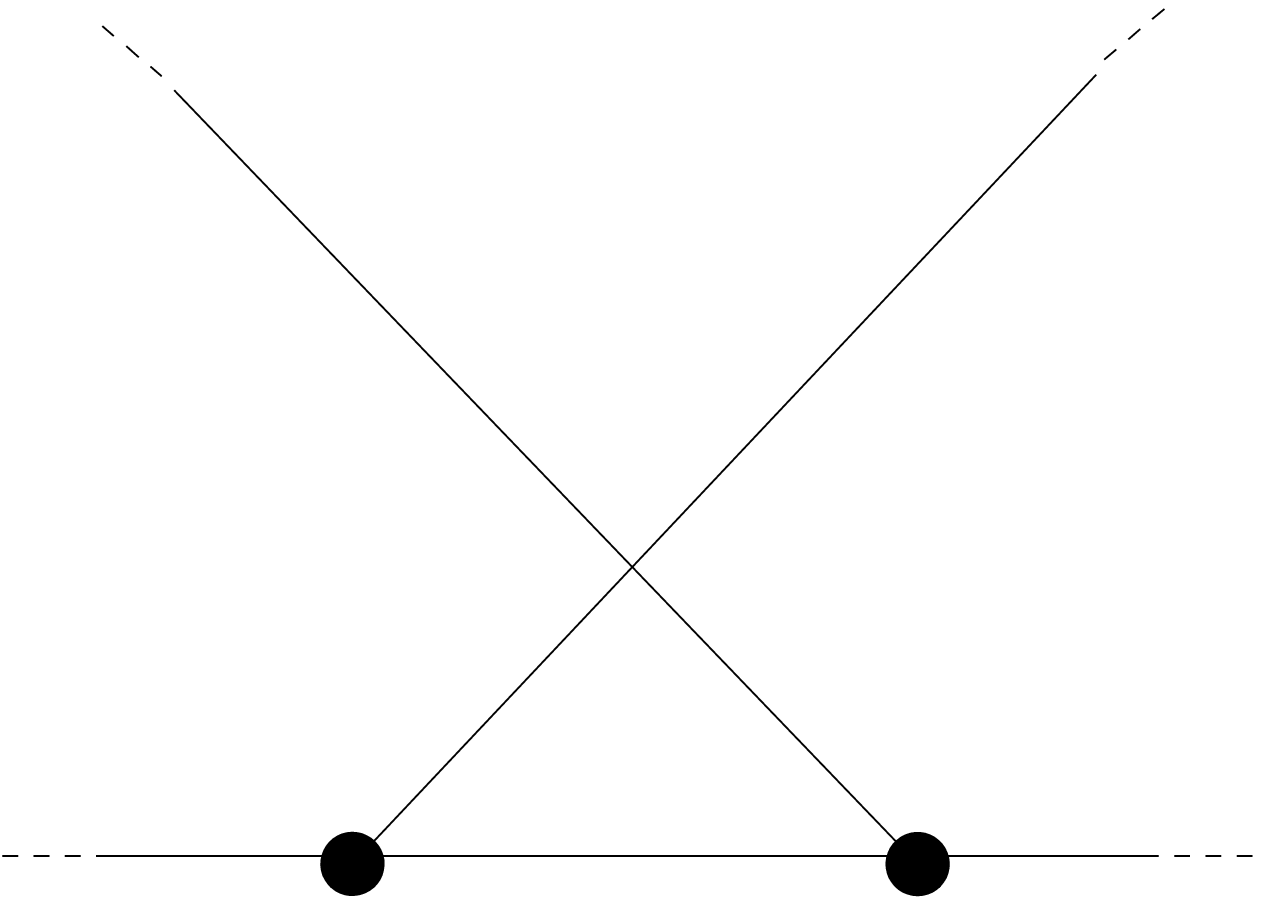}} 
\end{equation}
It is well known that $\W_m^k \cong (\TD_m^k/(STU, 1T))^*$; see Bar-Natan's paper \cite{BarNatan} for a proof.

It will be useful later to replace the Lie orientation by an ordering (labeling) of vertices and orientation of each edge.  In that case the STU relation is 
\begin{equation}
\label{LabeledSTU}
\raisebox{-3.5pc}{\includegraphics[height=7pc]{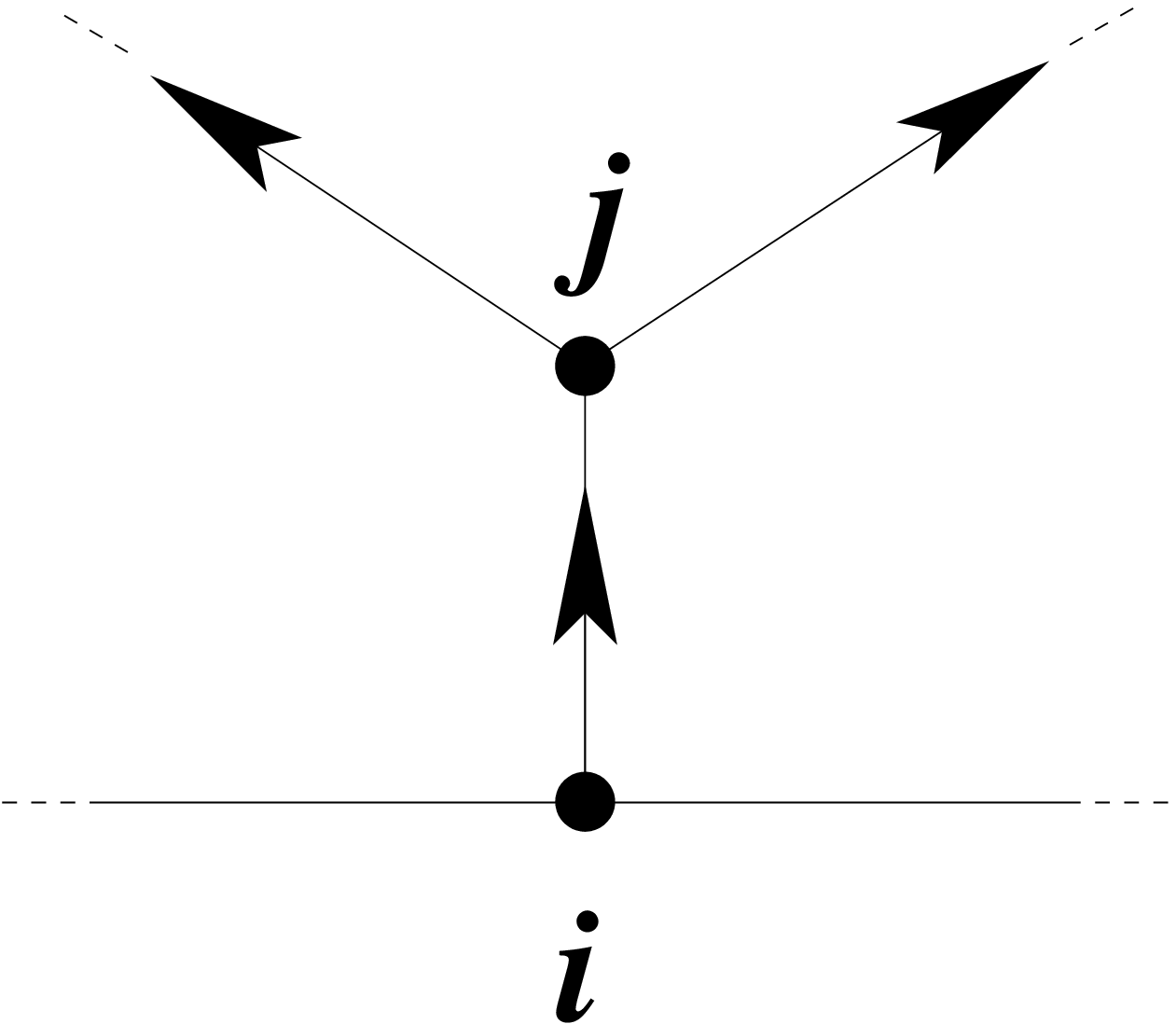}} =
\raisebox{-3.5pc}{\includegraphics[height=7pc]{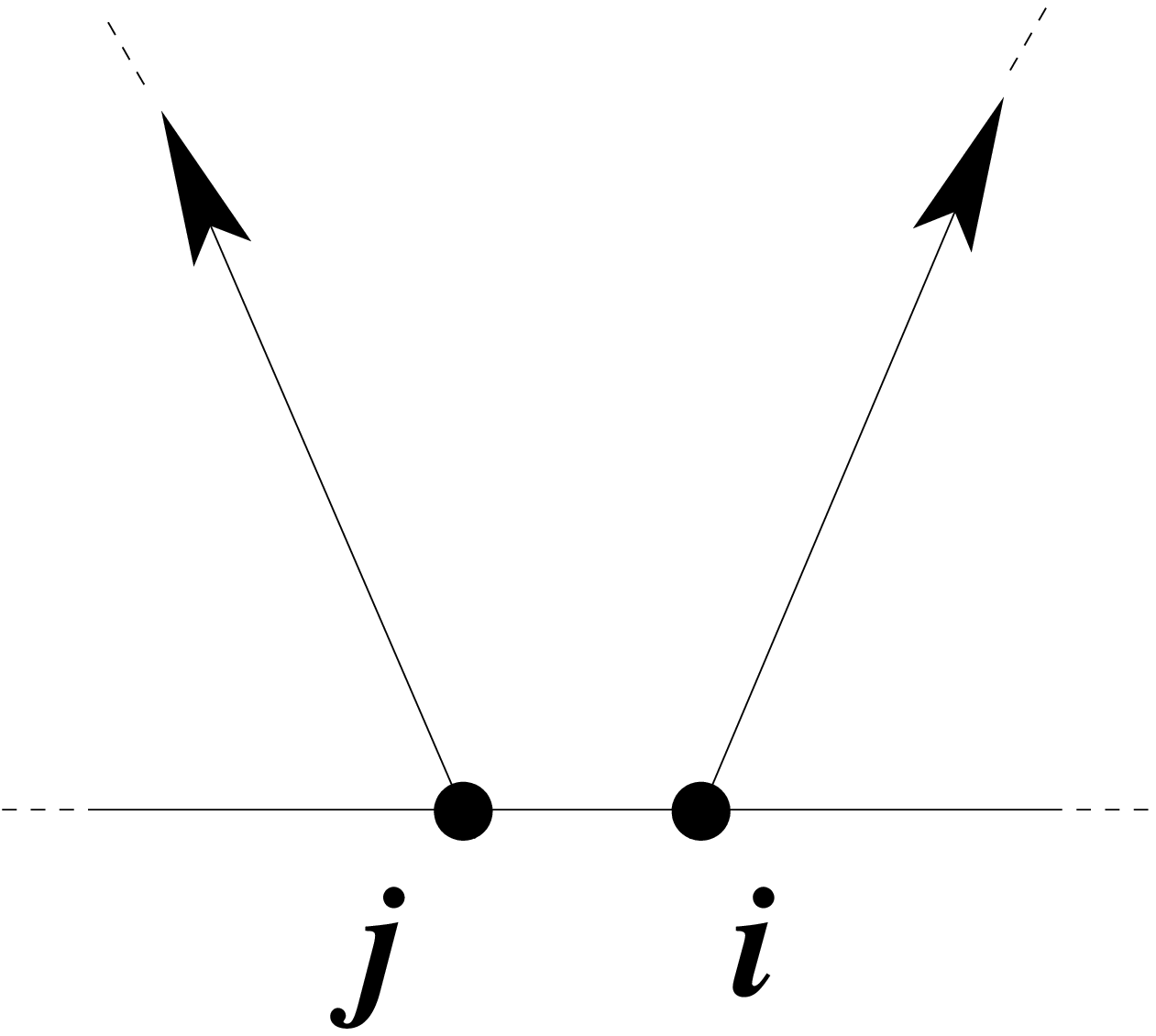}} -
\raisebox{-3.5pc}{\includegraphics[height=7pc]{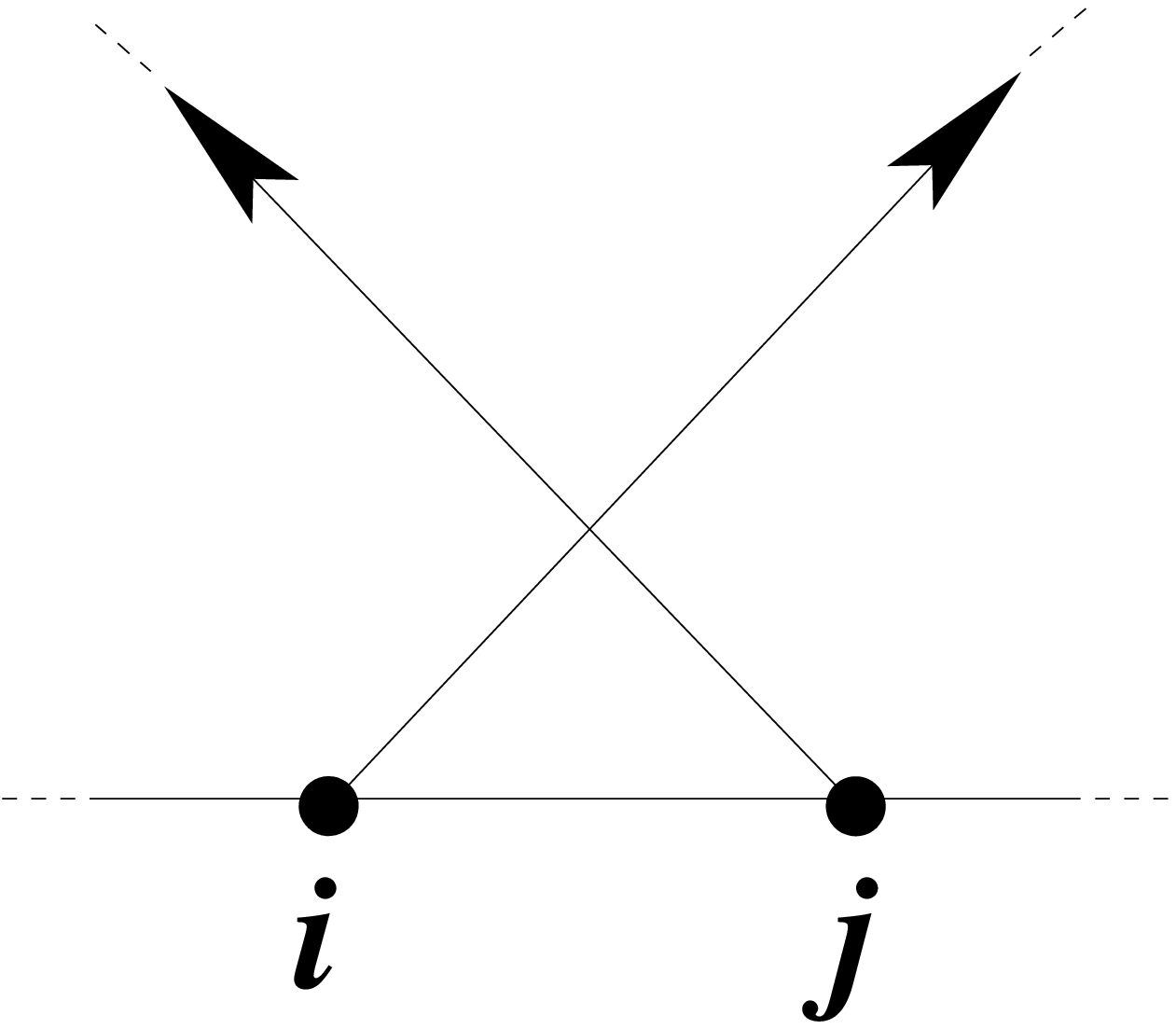}} 
\end{equation}
where the planar embedding does not matter.

The following originally appeared for closed knots in D Thurston's work \cite{Thurston} and was extended to string links and refined to give link-homotopy invariants in our paper \cite{MunsonVolicHtpyLinks}.  

\begin{theorem}{\cite{Thurston, VolicBT, MunsonVolicHtpyLinks}}
\label{BTfiniteType}
The map $\W^k_m \to \V^k_m/\V^k_{m-1}$ which is inverse to the map $w$ in (\ref{FiniteTypeIso}) can be constructed via 
\[
W \mapsto \frac{1}{(2m)!}\sum_{\mbox{labeled }D} W(D) I_D + \mbox{anomaly term}
\]
where the sum is taken over all trivalent diagrams $D$ with labels $\{1,...,2m\}$ on the vertices, and where $I_D$ is a configuration space integral depending on the labeled diagram $D$.
\end{theorem}

We define the integral $I_D$ in the next Subsection.  The anomaly term is 
related to the stratum in configuration space where all the points
have ``collided."  Its important feature here is that it vanishes for any link-homotopy invariant, such as the triple linking number.
The invariant above might depend on a choice of spherical forms used to define the integrals.  
We will address this issue in Sections \ref{ChoiceOfForms} and \ref{Indeterminacy}.  

Note that $(2m)!$ is precisely the number of labelings of a given isomorphism class of unlabeled (and unoriented) diagrams.  Thus the invariant in the Theorem above can equivalently be written as a sum over \emph{unlabeled} diagrams
\begin{equation}
\label{FiniteTypeInvtUnlabeled}
W \mapsto \sum_{\mbox{unlabeled }D} \frac{1}{|\Aut(D)|}W(D) I_D + \mbox{anomaly term}
\end{equation}
where $|\Aut(D)|$ is the number of automorphisms of $D$.  Another way of saying this is that we sum over elements of a basis for $\TD^k_m$.  In particular, each $D$ has an orientation (so that $W(D)$ is well defined), but any given diagram only appears with one Lie orientation in this sum.  The formulation in Theorem \ref{BTfiniteType} makes the association $D \mapsto I_D$ more immediate, but we find this latter formulation more tractable because there are fewer diagrams to sum over.

\subsection{Configuration space integrals}
\label{ConfigSpaceIntegrals}
In their paper \cite{Bott-Taubes}, Bott and Taubes constructed knot invariants via a bundle over the knot space.  
For string links (see our paper \cite{MunsonVolicHtpyLinks}), the bundle is  
\[
 \xymatrix{
F[q_1,...,q_k;t] \ar[r] & E[q_1,...,q_k; t]\ar[d] \\
 & \L_k = \Emb (\coprod_k \R, \R^3)}
\]
where the fiber $F[q_1,...,q_k; t]$ is a compactification of a configuration space of $q_1+\dotsb +q_k+t$ points in $\R^3$, $q_i$ of which lie on the $i^\mathrm{th}$ component of the link.  More precisely, the total space $E[q_1,...,q_k;t]$ is the pullback in the following square, and the map to $\L_k$ is the left-hand map followed by projection to $\L_k$.  
\begin{equation}
\label{BTsquare}
 \xymatrix{
E[q_1,...q_k;t] \ar[r] \ar[d] & C_{q_1+\dotsb +q_k+t}[\R^3]\ar[d] \\
\L_k \x  C_{q_1,...,q_k} \left[\coprod_1^k \R\right]  \ar[r] & C_{q_1+\dotsb +q_k}[\R^3]}
\end{equation}
The pullback square is explained below, and in Section \ref{MilnorIntegrals}, we will more carefully examine the fibers $F[q_1,...,q_k;t]$ in our particular example.

\subsubsection{Brief review of the Axelrod--Singer compactification}
Here $C_q[M]$ denotes the Axelrod--Singer compactification of the space $C_q(M)$ of configurations of $q$ points on a compact manifold $M$.  It is obtained by blowing up every diagonal in the product $M^q$.  
For $M=\R^n$, $C_q[\R^n]$ is considered as the subspace of $C_{q+1}[S^n]$ where the last point is fixed at $\infty$.  


A stratum of $C_q[M]$ is labeled by a collection $\{S_1,...,S_k\}$ of distinct subsets $S_i\subset \{1,...,q\}$ with $|S_i|\geq 2$ and satisfying the condition 
\[
S_i \cap S_j\neq \emptyset \Rightarrow \mbox{ either } S_i \subset S_j \mbox{ or } S_j \subset S_i  
\]
For each set $S_i$ in the collection, we can think of the points in $S_i$ as having collided.  If there is an $S_j \subset S_i$ in the collection, we can think of the points in $S_j$ as having first collided with each other and then with the remaining points in $S_i$.  Two strata indexed by $\{S_1,...,S_k\}$ and $\{S'_1,...,S'_j\}$ intersect precisely when the set $\{S_1,...,S_k, S'_1,...,S'_j\}$ satisfies the above condition.  In that case, that is the set which indexes the intersection.

Let $s_i=|S_i|$.  Roughly speaking, this compactification keeps track of the location of the ``collided point" as well as an ``infinitesimal configuration" or \emph{screen} for each $S_i$.  Such a screen is a point the space $(C_{s_i}(T_pM))/(\R^n \rtimes \R_+)$, where $\R^n \rtimes \R_+$ is the group of translations and oriented scalings of $\R^n \cong T_p M$.  From this one can verify that the stratum labeled $\{S_1,...,S_k\}$ has codimension $k$.  Consult the papers of Fulton and MacPherson \cite{Fulton-MacPherson} and Axelrod and Singer \cite{Axelrod-Singer} for more precise details.
\qed

Now it is clear that the right-hand vertical map in the square (\ref{BTsquare}) can be defined as projection to the first $q_1+\dotsb +q_k$ points.  The lower horizontal map (whose domain is defined below) 
comes from the fact that an embedding induces a map on configuration spaces.

\subsubsection{Compactifying the configurations of points on a string link}
\label{TechnicalForStringLinks}

Let $Q=q_1+\dotsb +q_k$, and define $C_{q_1,...,q_k} \left[\coprod_1^k \R\right] $ as the closure of the image of the map 
\[
\xymatrix{ C_{q_1}(\R) \x ... \x C_{q_k}(\R) \ar[r] & C_Q [\R^3] \left( \subset C_{Q+1} [S^3] \right)}
\]
induced by a fixed string link $L=(L_1,...,L_k)$.  Then the lower horizontal map in (\ref{BTsquare}) is certainly well defined.  To integrate over the fiber of the bundle $E[q_1,...,q_k;t]\to \L_k$, one needs to know that the closure $C_{q_1,...,q_k} \left[\coprod_1^k \R\right]$ has the structure of a manifold with corners.  Indeed, it inherits one from $C_{q_1+\dotsb +q_k+1} [S^3]$.  For a thorough proof, see 
\cite[Lemma 4.4]{MunsonVolicHtpyLinks}.  The basic idea of the proof is as follows.
Away from $\infty$, the local structure clearly is the same as in $C_q[\R]$, so the point is to describe a neighborhood of a configuration where some points have collided with $\infty$.  Such a neighborhood has various strata, points of which are described by collections of screens, just as in $C_Q[\R^3]$.  The only difference is that the tangent space to $\infty \in S^3$ is replaced by the subspace of lines (or rays) through its origin which correspond to the fixed directions of the string links components towards infinity.  (Thus the spaces of screens have lower dimension than in $C_Q[\R^3]$, but the same codimension.)
\qed

\begin{remark}[Notation] In the case of knots, the pullback $E[q_1,...,q_k;t]$ above reduces to a space $E[q;t]$ which was called $E_{q,t}$ in our previous work \cite{Rbo}; $C_{n,t}$ in the papers of Bott and Taubes \cite{Bott-Taubes} and Cattaneo et al
\cite{Cattaneo};
$C[q+t; \K] = C[q+t; \L_1]$ in our paper with Munson and Voli\'{c} \cite{MunsonVolicHtpyLinks}, and $ev^*C_t(\R^3)$ 
in work of Pelatt \cite{Pelatt}.  We have tried to use the shortest notation that distinguishes between this total space and the fiber of the associated bundle.  
\end{remark}

One can show that the fiber $F[q_1,...,q_k;t]$ is a manifold with corners \cite[Proposition 4.6]{MunsonVolicHtpyLinks}.  This fiber is also orientable.  By fixing an orientation on $\R^3$ and using the orientation of the link, the fibers become \emph{oriented}.  

Next, one integrates differential forms along the fiber of $E[q_1,...,q_k; t] \to \L_k$ 
Specifically, the forms are obtained from the following maps:
\begin{gather*}
\phi_{ij}\co C_r (\R^3)  \to S^2, \qquad 1\leq i<j \leq r \\
\phi_{ij}(x_1,...,x_r) = \frac{x_j - x_i}{|x_j - x_i|}
\end{gather*}
This induces 
a map  $\phi_{ij}\co E(q_1,...,q_k; t) \to S^2$ for every pair $(i<j)$.  Let $\omega$ be a 2--form on $S^2$ representing a generator of $H^2(S^2)$ and let $\theta_{ij} = \phi_{ij}^* \omega$.  

We can now define the $I_D$ in Theorem \ref{BTfiniteType}.  Given a trivalent diagram $D \in \TD_m^k$ with labels $\{1,...,2m\}$ on the vertices, let $q_i$ be the number of vertices on the $i$th interval, and let $t$ be the number of free vertices.  The function $I_D$ is an integral along the fiber $F[q_1,...,q_k;t]$ of the bundle $E[q_1,...,q_k;t]\to \L_k$, given by
\[
I_D := \int_{F[q_1,...,q_k;t]} \theta_D := \int_{F[q_1,...,q_k;t]} \theta_{i_1 j_1} \cdots \theta_{i_e j_e}
\]
where $(i_1<j_1),...,(i_e<j_e)$ are the pairs of endpoints of edges in $D$.  Since these forms are even-dimensional, their order in the product is irrelevant.  For example, to the diagram 
\[
\includegraphics[height=6pc]{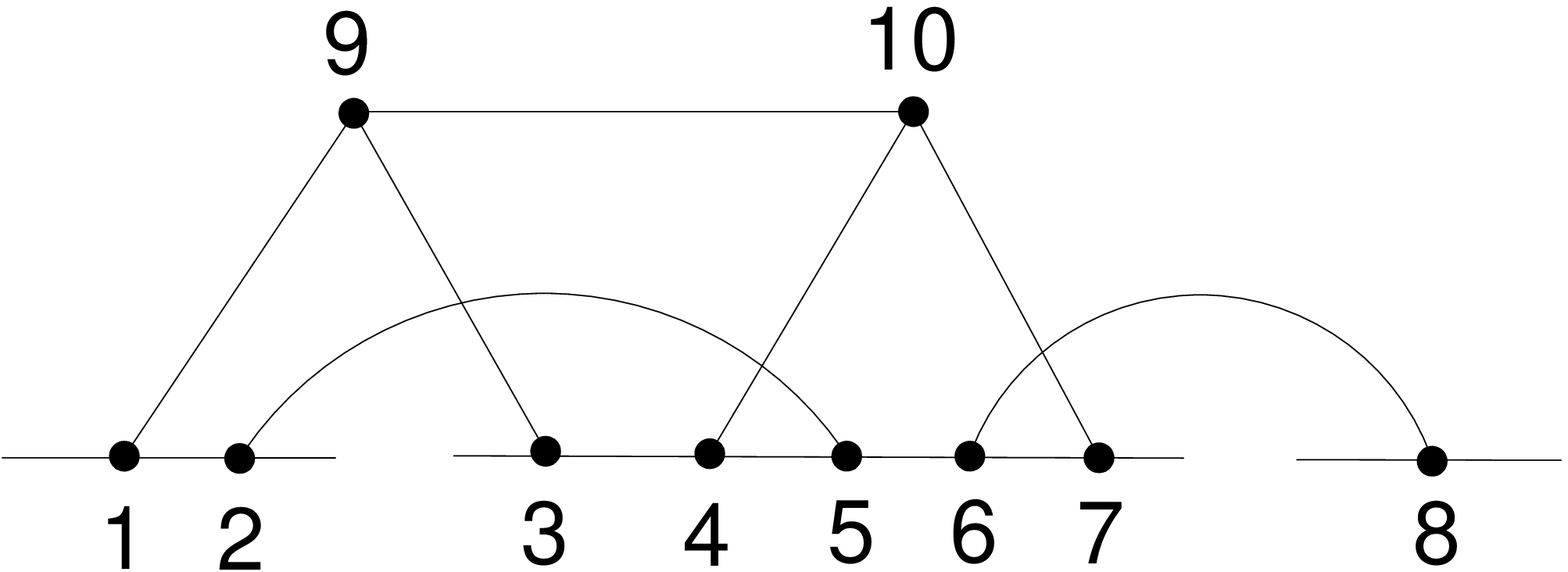}
\]
we associate the following integral:
\[
\int_{F[2,5,1;2]} \theta_{19}\theta_{25}\theta_{39}\theta_{4,10}\theta_{68}\theta_{7,10}\theta_{9,10}
\]
Since a Lie orientation is equivalent to a vertex-ordering and orientation of every edge, there is a well defined way of associating an integral to an unlabeled Lie-oriented diagram.


\begin{remark}
In the existing literature, each $\omega_i$ is often rotation-invariant or at least symmetric with respect to the antipodal map.  Because of the particular arguments we invoke, we will instead find it useful to choose different non-symmetric forms $\omega_i$.  Thus we will eventually need to keep track of the order of the $S^2$ factors in the product above.
\end{remark}

Since both the dimension of $F=F[q_1,...,q_k;t]$ and the degree of $\theta_D$ (= twice the number of edges in $D$) is $q_1+\dotsb +q_k + 3t$, the integral $I_D$ is a 0--form, ie, a function.  

In general each $I_D$ is not closed, but certain linear combinations of various $I_D$ are closed.  
In fact, since the fiber has nonempty boundary, Stokes' theorem implies that for any form $\alpha$
\begin{equation}
\label{Stokes}
d \int_F \alpha = \int_F d \alpha \pm \int_{\d F} (\alpha|_{\d_F E}).
\end{equation}
If $\alpha$ is a linear combination of products of $\theta_{ij}$, then $\alpha$ is closed, so showing that $\int_F \alpha$ is closed amounts to showing that the integral of the restriction of $\alpha$ to $\d_F E \subset E$ vanishes.  

Part of the content of Theorem \ref{BTfiniteType} is that $\Sigma_D W(D) I_D$ is closed, ie, a locally constant function, or link invariant (up to an ``anomaly term," which vanishes for  link-homotopy invariants). We will review the proof of this for the triple linking number in Section \ref{VanishingArguments}.

\subsection{Milnor triple linking number for string links}
\label{MilnorInvariants}
Finally, we review the adaptation of Milnor's triple linking number \cite{MilnorLinkGroups} to the case of string links.  
The triple linking number is defined by studying the lower central series of the fundamental group of the link complement.  Let $L$ be (the image of) a $k$-component string link in $I \x D^2$, and let $\pi = \pi_1(I \x D^2 \setminus L, p)$.  To show that our invariant agrees with the triple linking number, we will need to choose basepoint, so let $p=(0,0,1)$. 
Let $\pi^1=\pi$, and let $\pi^n = [\pi, \pi^{n-1}]$, the $n^\mathrm{th}$ term in the lower central series.  From the Wirtinger presentation of a link, one sees that the meridians and all their conjugates generate $\pi$.  Milnor proved \cite[Theorem 4]{MilnorIsotopy} that the meridians themselves generate $\pi/\pi^n$ for all $n$.  Thus if $F_k$ denotes the free group on generators $a_1,...,a_k$, there is a map $F_k/F_k^n \to \pi/\pi^n$ which is surjective, and in fact an isomorphism; see also the work of Habegger and Lin \cite{HabeggerLin}.  Thus there is a word in the $a_i$ which maps to the $j^\mathrm{th}$ longitude $[\ell_j] \in \pi/ \pi^n$.  (Note that an orientation of the link gives \emph{canonical} longitudes and thus \emph{canonical} meridians, by choosing each meridian to have linking number +1 with the corresponding longitude.)

The \emph{Magnus expansion} is a group homomorphism from $F_k$ to the multiplicative group in the power series ring in non-commuting variables $t_1,...,t_k$; it is given by $a_i \mapsto 1+t_i$ and $a_i^{-1}\mapsto 1 - t_i +t_i^2 -...$.  The \emph{Milnor invariant} $\mu_{i_1,...,i_r; j}$ is then defined as the coefficient of $t_{i_1}\cdots t_{i_r}$ in the Magnus expansion of $[\ell_j] \in \pi/\pi^n$.  It is straightforward to check that this is well defined for $n>r$.  
The invariants $\mu_{i;j}$ are just the pairwise linking numbers.  For three-component string links, there are six different ``triple linking numbers", but by permuting labels on the strands it suffices to study just one of them.  Thus we take $\mu_{123}\equiv \mu_{12;3}$ as the triple linking number for string links. 

Milnor's original definitions were for closed links, but in that case the $\mu$ invariants are only well defined modulo the gcd of the lower order invariants (ie, pairwise linking numbers in the case of $\mu_{123}$).  
These invariants of closed links have also received recent attention, for example in work of Mellor and Melvin \cite{MellorMelvin}, as well as in work of DeTurck, Gluck, Komendarczyk, Melvin, Shonkwiler, and Vela-Vick\cite{SixAuthorsShort}. 
We restrict our attention to Milnor invariants of string links (as studied by Habegger and Lin \cite{HabeggerLin}), where they are well defined integers.

\section{The configuration space integral for the triple linking number}
\label{MilnorIntegrals}

In this Section, we examine a configuration space integral formula for the triple linking number $\mu_{123}$ for string links.  
We first use a formula of Mellor to find an expression of the triple linking number for string links in terms of configuration space integrals.  Then we examine in detail the compactifications of configuration spaces which appear in these integrals.  This will be useful in our description of the invariant as a degree.   Then we examine how the integrals combine to make a link invariant.  This step is essentially the proof of the isotopy invariance of Theorem \ref{BTfiniteType} 
for the triple linking number.  
Going through this proof is also important because it will show us how to glue and collapse boundaries of the configuration spaces in our 
construction.  

\subsection{The weight system for the triple linking number}
\label{MilnorWeightSystem}
It is known that for string links, each Milnor invariant $\mu_{i_1...i_r; j}$ is finite-type of type $r$; see the papers of Bar-Natan \cite{BNHtpyLinks} and Mellor \cite{MellorMilnorWeight}.  The triple linking number $\mu_{123}\equiv\mu_{12;3}$ has an associated weight system $W_{123}$, given by the map (\ref{CanonWeightSystem}).  
Recall that there is no anomaly term for weight systems which yield link-homotopy invariants.  (We will verify this fact in this example in this Section.)  Then the formulation (\ref{FiniteTypeInvtUnlabeled}) of Theorem \ref{BTfiniteType} says that
\[
\sum_{\mbox{unlabeled }D} \frac{1}{|\Aut(D)|} W_{123}(D) I_D 
\]
is an element of $\V_2$ which in $\V_2/\V_1$ agrees with the image of the triple linking number.
So from now until the end of Section \ref{MilnorIntegrals} we will let $\mu_{123}\in \V_2$ denote the right-hand side of the above equality.  
We will check in Section \ref{Indeterminacy} that $\mu_{123}$ is 
the bona fide triple linking number, thus excusing our abuse of notation.

We have $W_{123}\in \W_2^3 \subset (\CD_2^3)^*$, a functional on diagrams with four vertices.  Since the triple linking number is an invariant of link-homotopy, we know that $W_{123}$ can be nonzero only on diagrams where no chord joins vertices on the same strand; see our paper with Munson and Voli\'{c} \cite[Definition 3.27, Theorem 5.8]{MunsonVolicHtpyLinks}.  We may furthermore consider only diagrams where each of the three strands has at least one vertex, since the triple linking number cannot detect crossing changes between only two components.  There are seven such diagrams spanning\footnote{These seven diagrams satisfy relations in $\CD_2^3/(4T)$, but the calculation below is easy enough that we do not need to use these relations.} a subspace of $\CD_2^3$:  the diagrams $L,M,R$ \\
\[
\includegraphics[height=1.7pc]{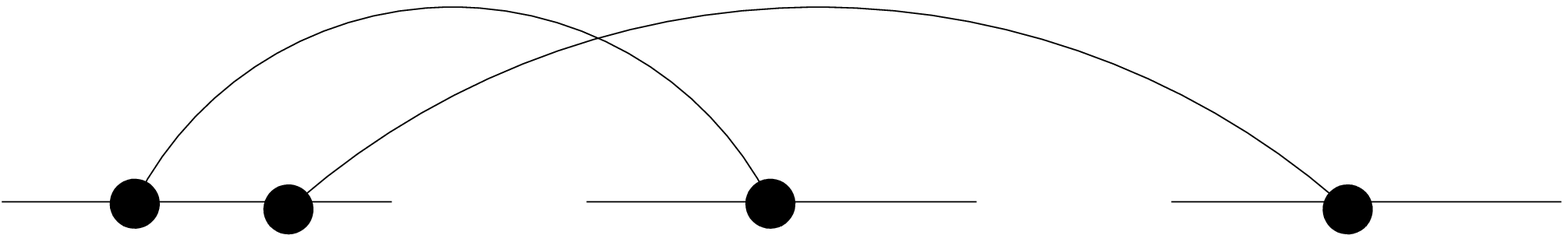},
\includegraphics[height=2pc]{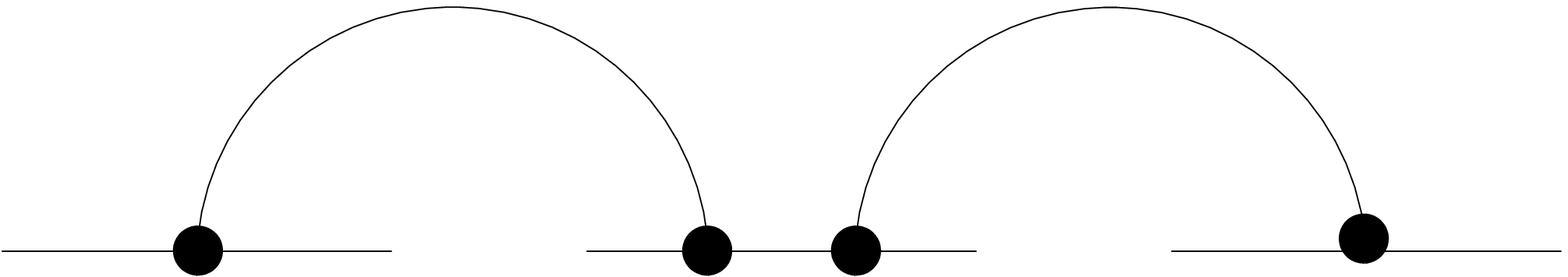},
\includegraphics[height=1.7pc]{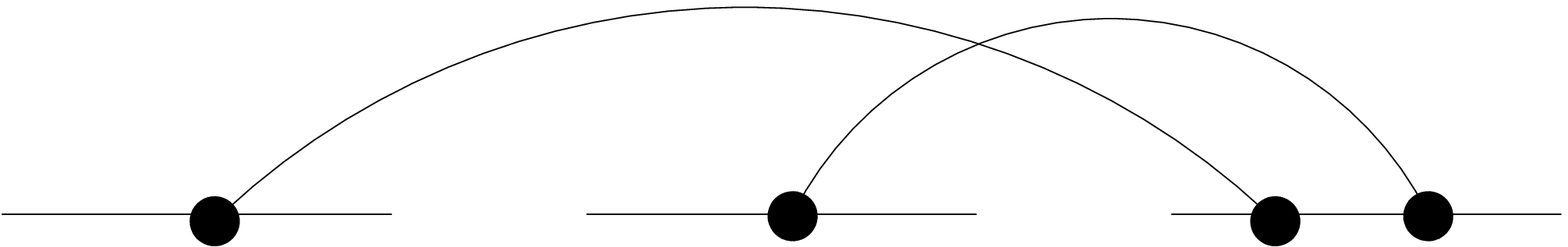}
\]
the diagrams $L',M',R'$ 
\[
\includegraphics[height=3pc]{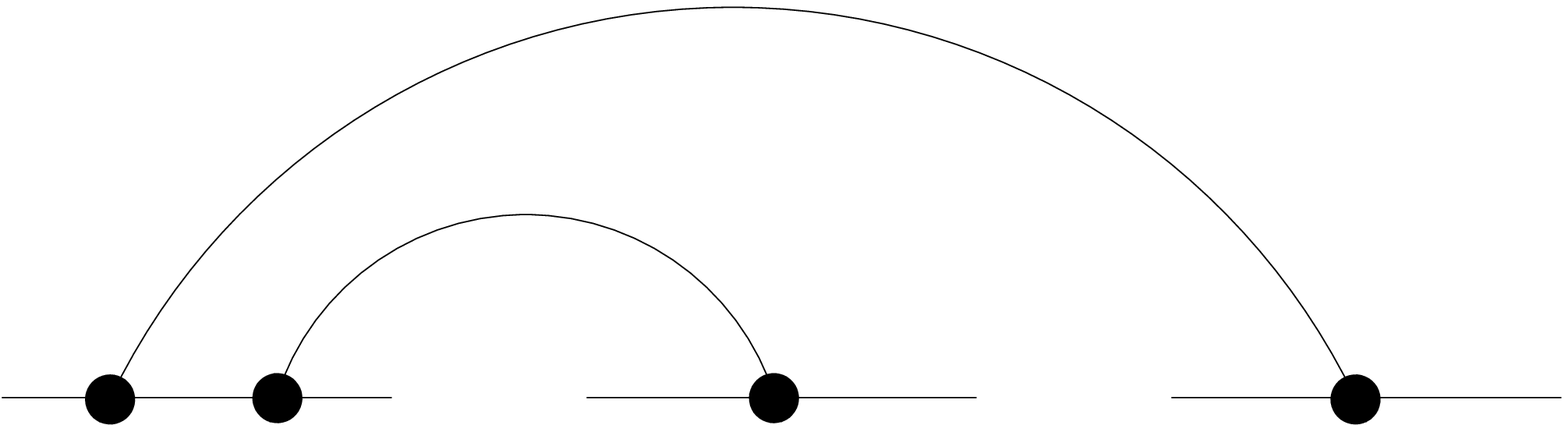},
\includegraphics[height=2pc]{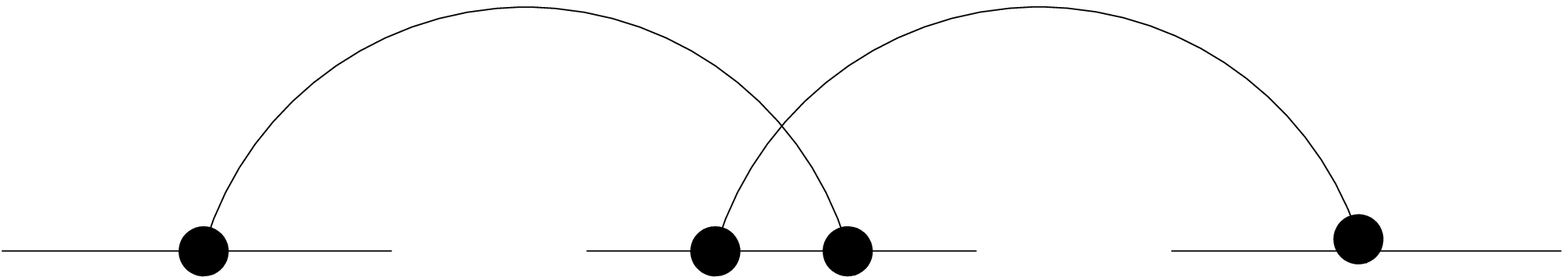},
\includegraphics[height=3pc]{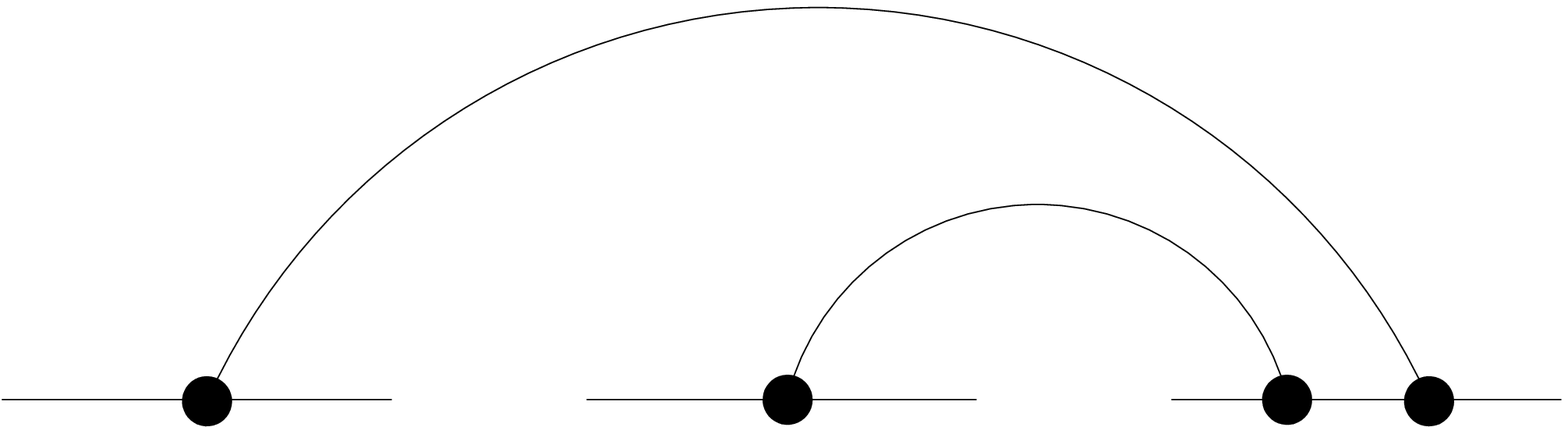}
\]
and the diagram $T$
\[
\includegraphics[height=3pc]{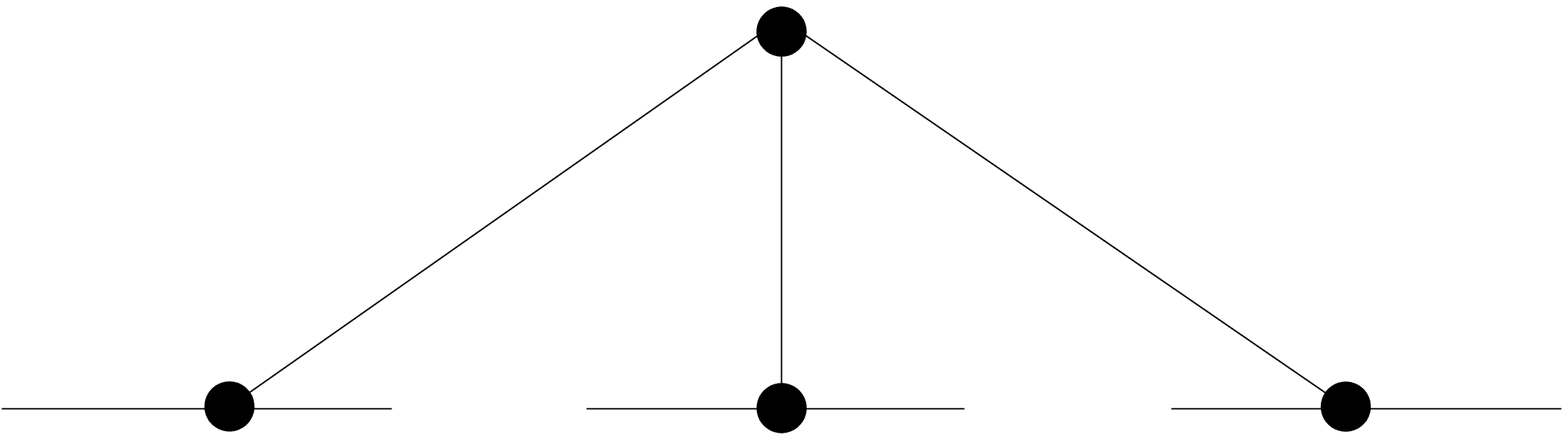}.
\]
Since the strands are labeled and oriented, each of these diagrams has no nontrivial automorphisms.  Thus we have so far
\[
\mu_{123} = \sum_{D\in \{L,M,R,L'M',R',T\}} W_{123}(D) I_D.
\]
Using a skein relation for the Milnor invariants proven by Polyak \cite{PolyakMilnorSkein}, Mellor proved a recursive formula for the weight systems of these invariants.  This formula is Theorem 2 of Mellor's paper \cite{MellorMilnorWeight}, to which the reader may refer for details.  
We apply Mellor's formula to the weight system $W_{123}$ and the six chord diagrams $L,M,R,L',M',R'$.  The result is that $W_{123}$ vanishes on $L',M',$ and $R'$, and its values on $L,M,R$ are $1,-1,1$.

The STU relation will ultimately tell us the value of $W_{123}$ on $T$.  But first, to clarify the association $D \mapsto I_D$ in this setting of unlabeled diagrams, the canonical Lie orientation on a chord diagram corresponds to a canonical ordering of the vertices (from left to right) and a canonical orientation of the edges (also from left to right).  This defines $I_D$ for the chord diagrams $D$.  For the diagram $T$, we can choose any vertex-ordering and edge-orientations and then account for the sign by using the labeled STU relation (\ref{LabeledSTU}) to compute $W_{123}(T)$.  We choose to label the free vertex ``4" and orient the edges from smaller to larger index.  

\begin{remark} 
If one considers the unlabeled STU relation (\ref{UnlabeledSTU}) featuring the unlabeled, Lie-oriented basis element $T$, one gets the same sign for $W_{123}(T)$ as for the labeling we chose.  Thus under the correspondence between the two types of orientation, our choice of labeling of $T$ gives the same orientation as the unlabeled, Lie-oriented basis element $T$.
\end{remark} 

The calculation described above yields the following result.

\begin{corollary}[to Theorem 2 of Mellor's paper \cite{MellorMilnorWeight} and Theorem \ref{BTfiniteType}]
\label{MellorCor}
The triple linking number for string links is, up to a type--1 invariant, the sum of integrals
\[
\mu_{123} =  I_L - I_M + I_R - I_T 
\]
where the terms $I_D$ above are the integrals over compactified configuration spaces associated to the diagrams $D=L,M,R,T$:
\begin{align*}
I_L = \int_{F[2,1,1;0]} \theta_{13}\theta_{24} & &
I_M= \int_{F[1,2,1;0]} \theta_{12}\theta_{34}\\
I_R=\int_{F[1,1,2;0]} \theta_{13}\theta_{24} & &
I_T= \int_{F[1,1,1;1]} \theta_{14}\theta_{24}\theta_{34}
\end{align*}
\end{corollary}

We will see in Section \ref{Indeterminacy} that we can remove the ``up to a type--1 invariant" indeterminacy by making appropriate choices in defining these integrals.

\subsection{The compactified configuration spaces}
\label{CptfdConfigSpaces}
In our topological construction, it will be useful to understand in detail the compactifications of these four configuration space fibers, which we will sometimes abbreviate as $F_L, F_M, F_R, F_T$.  \\

\subsubsection{The space $F_T$}
\label{TheSpaceFT}
First we describe $F_T=F[1,1,1;1]$, the fiber of the bundle $E[1,1,1;1]$ over a 3-component string link $L$. This space compactifies $F(1,1,1;1)=\{ (x_1,...,x_4) \in C_4(\R^3) | x_1,x_2,x_3 \in \im \>L \}$.  By forgetting  $x_4$ which is free to roam in $\R^3$, the uncompactified $F(1,1,1;1)$ maps to $\R\x \R \x \R$ as a fiber bundle with fiber homeomorphic to $\R^3 \setminus \{\mbox{3 $D^3$'s}\}$, the complement of three closed 3-balls in $\R^3$.  The base is contractible, so $F(1,1,1;1)\cong \R \x \R \x \R \x (\R^3 \setminus \{\mbox{3 $D^3$'s}\})$. 

The compactification $F[1,1,1;1]$ is a subset of the Axelrod--Singer compactification $C_4[\R^3]\subset C_5[S^3]$.  Since $F[1,1,1;1]$ is a manifold with corners, it has a collar neighborhood of the boundary and is homotopy-equivalent to its interior, $F(1,1,1;1)$.  In $F[1,1,1;1]$ the subsets which index codimension--1 faces are 
\[
\{x_1,x_4\}, \{x_2,x_4\}, \{x_3,x_4\}
\] 
as well as the $15$ sets
\[
\{x_1, \infty\}, \{x_2,\infty\},\{x_3, \infty\}, \{x_4,\infty\},..., \{x_1,x_2,x_3,x_4, \infty\}
\]
consisting of a nonempty subset of the $x_i$ together with $\infty$.  If $S$ is a set of the second type, the corresponding face has $2^{a(S)}$ connected components, where $a(S)= \#(S \cap \{x_1,x_2,x_3\})$.  

For any of these compactified configuration spaces, we call a face whose indexing set consists of precisely two $x_i$ a \emph{principal face}.  A face whose indexing set contains more than two $x_i$ is called a \emph{hidden face}.  We call a face whose indexing set contains the symbol $\infty$ a \emph{face at infinity} or, for reduced verbiage, just an \emph{infinite face}.  In all the spaces which we consider here, every face is either principal or infinite.

An explicit description of an infinite face will be useful in the proof of Lemma \ref{InfiniteVanish} below.  In $C_n[\R^3]$, the face where $m$ of the $n$ points have escaped to infinity can be described as the product $C_{n-m}(\R^3) \x (C_{m+1}(\R^3)/(\R^3 \rtimes \R_+))$ where $\R^3 \rtimes \R_+$ is the group of translations and scalings of $\R^3$ (cf.~ Turchin's paper \cite[Section 6, page 18]{VictorContextFree}).  The $(m+1)^\mathrm{st}$ point in the second factor is a ``fat point" of points which have not escaped to infinity.  If we translate this point to the origin, then we can rewrite such a face as $C_{n-m}(\R^3) \x (C_m(\R^3 \setminus \{0\})/\R_+)$.  (Alternatively, $\R^3 \setminus \{0\}$ could be replaced by the complement of a ball around the origin.)  Note that this description applies even when $m=n$.  In $F[1,1,1;1]$ an infinite face has a similar description, except that the $x_i$ are constrained to lie on the fixed linear embedding.

\subsubsection{The spaces $F_L, F_M, F_R$}
Now we describe the fiber $F_L=F[2,1,1;0]$ of the bundle $E[2,1,1;0]$, which compactifies $F(2,1,1;0)=\{(x_1,...,x_4) \in C_4(\R\sqcup \R \sqcup \R)| (x_1<x_2)\in C_2(\R), x_3\in \R, x_4\in \R\}$.  The latter space is homeomorphic to an open 4-ball.  As in the case of $F_T$ above, the compactification has a collar neighborhood of the boundary, and is homotopy-equivalent to its interior $F(2,1,1;0)$.

The subsets which index codimension--1 faces of $F[2,1,1;0]$ are $\{x_1,x_2\}$ and the 15 sets consisting of a nonempty subset of $\{x_1,...,x_4\}$ together with $\infty$.  The compactification $F[2,1,1;0]$ is a manifold with corners, but topologically it is homeomorphic to a closed ball.  As with $F_T$, an infinite face may have multiple components, and the description of infinite faces in the previous Subsection apply equally well here.

The spaces $F_M=F[1,2,1;0]$ and $F_R=F[1,1,2;0]$ are constructed in a similar manner and are diffeomorphic to $F_L=F[2,1,1;0]$.

\subsection{Vanishing Arguments}
\label{VanishingArguments}

Having examined the compactifications in our example, we now examine how the configuration space integrals combine to form a link invariant.  We first take a small step towards our gluing construction by rewriting all four integrals $I_D$ as integrals over 6--dimensional fibers.  For each $D=L,M,R$, we replace the integral along the fiber of $E_D \to \L_3$ by the integral along the fiber of $E_D \x S^2 \to \L_3$, where the original integrand is multiplied by form $\omega$ on $S^2$ with unit volume:
\begin{align*}
I_L = \int_{F[2,1,1;0] \x S^2} \theta_{13}\theta_{24}\omega & &
I_M= \int_{F[1,2,1;0] \x S^2} \theta_{12}\theta_{34}\omega \\
I_R=\int_{F[1,1,2;0] \x S^2} \theta_{13}\theta_{24}\omega & &
I_T= \int_{F[1,1,1;1]} \theta_{14}\theta_{24}\theta_{34}
\end{align*}
It will also be convenient to transpose some of the indices in the $\theta_{ij}$ as below, where we use the fact that $\theta_{ij}=-\theta_{ji}$.  The choice of transpositions will be explained later (see Remark \ref{OrderOfIndices} and Figure \ref{CollidedDiagrams}).
\begin{align*}
I_L = \int_{F_L \x S^2} \theta_{31}\theta_{42}\omega & & 
I_M= \int_{F_M \x S^2} \theta_{12}(-\theta_{43})\omega \\
I_R=\int_{F_R \x S^2} \theta_{13}\theta_{24}\omega & & 
I_T= \int_{F_T} \theta_{14}\theta_{24}\theta_{34}
\end{align*}

The more general statement of the lemma below is well known for knots, though there is an extra subtlety for string links because of the nature of some additional infinite faces \cite[Remark 4.32]{MunsonVolicHtpyLinks}.  However, the description of the infinite faces at the end of \ref{TheSpaceFT} makes it clear that this lemma holds for them too.  
The transpositions of some of the indices in the $\phi_{ij}$ below follow those in the integrals above.

\begin{lemma}
\label{MapsPhi}
There are smooth maps $\Phi_D$ as follows:
\begin{align*}
\Phi_L = (\mathrm{proj}_{S^2}, \phi_{31}, \phi_{42}) \co E_L \x S^2 \to S^2 \x S^2 \x S^2 \\
\Phi_M = (\phi_{12}, \mathrm{proj}_{S^2}, \phi_{43}) \co E_M \x S^2 \to S^2 \x S^2 \x S^2 \\
\Phi_R = (\phi_{13}, \phi_{24}, \mathrm{proj}_{S^2}) \co E_R \x S^2 \to S^2 \x S^2 \x S^2 \\
\Phi_T = (\phi_{14}, \phi_{24}, \phi_{34}) \co E_T \to S^2 \x S^2 \x S^2
\end{align*}
extending the maps $\phi_{ij}\co (x_1,...,x_4)\mapsto \prod (x_j - x_i)/|x_j - x_i|$ 
 from the interior of $E_D$ to all the boundary faces of $E_D$.
\qed
\end{lemma}


Next, a configuration space integral is only a link invariant because the integrals along the boundary faces vanish or cancel each other.  (The necessity of this condition comes from Stokes' Theorem, equation (\ref{Stokes}).)  The next lemmas are important because they indicate how we should collapse and glue together configuration spaces.  

\begin{lemma}
\label{PrincipalCancel}
The integrals along the three principal faces of $F_T$ cancel with those along the three principal faces of $F_L \x S^2, F_M \x S^2,$ and $F_R \x S^2$.
\end{lemma}
\begin{lemma}
\label{InfiniteVanish}
The integral along any infinite face of $F_D (\x S^2)$ vanishes, where $D$ is any of $L,M,R,T$.
\end{lemma}
\begin{proof}[Sketch proof of Lemma  \ref{PrincipalCancel}]
The main idea of the proof (see the work of D Thurston \cite{Thurston} and Voli\'{c} \cite{VolicBT}) is to use the STU relation and the fact that the faces corresponding to the terms in the relation (and the restrictions of the integrands to those faces) are identical, at least up to a sign.  In this case, Mellor's formula says that $W_{123}$ vanishes on $L', M', R'$.  Thus, with our vertex-orderings and edge-orientations on $L,M,R,T$, the STU relations are
\begin{align*}
W_{123}(T) = -W_{123}(L)  & &
W_{123}(T) = W_{123}(M) & &
W_{123}(T) = -W_{123}(R),
\end{align*}
showing how to pair together these six faces.  Each principal face of $F_T$ is diffeomorphic to $S^2 \x F[1,1,1;0]$; the same is true of each principal face of $F_D \x S^2$.  Similarly, the restrictions of the forms to be integrated agree, up to a sign, on each pair of faces.  One checks that the signs coming from the forms, the orientations of the faces, and the coefficients $W_{123}(D)$ above are such that each of the three pairs of integrals along principal faces cancels.
We will give more explicit details when we recast this topologically in Sections \ref{GlueBundlesLinkSpace} and \ref{ProofThm1}.
\end{proof}

\begin{proof}[Proof of Lemma \ref{InfiniteVanish}]
The idea of the proof of this lemma is also present in the above references. 
The main point is that the image of each infinite face under $\Phi_D$ is a 
subspace of $S^2\x S^2 \x S^2$ of codimension at least 1.
This allows us to choose a form on $S^2 \x S^2 \x S^2$ which generates its top (integral) cohomology, but which is supported outside of the image of the infinite faces under $\Phi_D$.  Then $\theta_D$, its pullback via $\Phi_D$, is zero.  Below are details for all the cases:
\begin{itemize}
\item[(a):] 
Suppose $\SS$ is an infinite face of $E_D$ such that the indexing set $S$ does not contain a free vertex.  Let $x_i\in S$ and let $x_j$ be the point connected to $x_i$ by a chord or edge.  We claim the image $\phi_{ij}(\SS)$ is contained in a positive-codimension subspace $C$ of $S^2$.  

First suppose $x_j$ is an interval vertex.  In particular, if ${v}_i$ and ${v}_j$ are the directions of the fixed linear maps for the strands containing $x_i$ and $x_j$, then $C$ is the great circle which is the intersection of $\mathrm{span}\{{v}_i, {v}_j\}$ with the unit sphere.  (With our choices of ${v}_i$, $C$ is the equator in the $xy$-plane.)

If $x_j$ is the free vertex $x_4$ in the diagram $T$ and $x_j \notin S$, then $\phi_{i4}(\SS)$ is the point which is the intersection of $\mathrm{span}\{{v_i}\}$ with $S^2$.

\item[(b):]
Now suppose the indexing set $S$ of $\SS$ contains a free vertex.  In other words $\SS$ is an infinite face of $E_D$ with $D=T$, and $x_4 \in S$.  

First, if two (of the remaining three) points $x_i, x_j$ are \emph{not} in $S$, then $\phi_{i4} \x \phi_{j4}(\SS)$ is contained in the diagonal of the product of the two appropriate factors of $S^2$ in $S^2 \x S^2 \x S^2$.  

Now suppose that there are two other points $x_i, x_j \in S$.  If the last remaining point $x_k$ does not escape to infinity, then $x_4 - x_k$ does not depend on $x_k$, but only on the direction of escape of $x_4$.  So in this case  $\Phi_T$ factors through the screen-space of the three points escaping to infinity.  Since this space has dimension $3+1+1 - 1=4$, $\Phi_D(\SS)$ is at most 4-dimensional.

Finally, if all the points escape to infinity, the map $\Phi_T$ again does not depend on the link and factors through the screen-space, which in this case has dimension $3+1+1+1 -1=5$.
\end{itemize}
\end{proof}

\begin{remark}
This lemma is also proven in our paper \cite[Proposition 4.31]{MunsonVolicHtpyLinks}.  However, if one traces through that general proof in this example, one finds that some of these infinite faces are treated using a ``symmetry" argument (the hidden face involutions, as in work of Kontsevich \cite{KontsevichFeynman}, Bott and Taubes \cite{Bott-Taubes}, D Thurston \cite{Thurston}, and Voli\'{c} \cite{VolicBT}) rather than a ``degeneracy" argument.  We prefer to use only degeneracy arguments in this example because this simplifies our gluing construction.  An earlier draft of this paper used a more complicated setting of string links for this purpose, though that setting turned out not to be strictly necessary for avoiding the use of hidden face involutions.
\end{remark}

We now partially describe the image of all the infinite faces, which we denote $\DD$ for ``degenerate locus.".

\begin{proposition}[The degenerate locus]
\label{DegenerateLocus}
Let $C\subset S^2$ be the equator in the $xy$-plane, $C_+$ the half of $C$ with positive $x$-coordinate, and $C_-$ the half of $C$ with negative $x$-coordinate.
Let $H_+, H_-$ denote the closed upper and lower hemispheres of $S^2$.
Let $\DD$ be the union of the images of all the infinite faces under $\Phi_L, \Phi_M, \Phi_R, \Phi_T$. 
Then $\DD$ is a codimension--1 subspace contained in the subspace $\widetilde{\DD}$, which is defined as the union of 
\begin{align}
\label{GreatCircles}
(C_+ \x S^2 \x S^2) \cup (S^2 \x C \x S^2) \cup (S^2 \x S^2 \x C_-)
\end{align}
together with 
\begin{align}
\label{Hemispheres}
(H_+ \x H_+ \x H_+) \cup (H_- \x H_- \x H_-).
\end{align}
\end{proposition}

\begin{proof}
First consider faces of type (a) as in Lemma \ref{InfiniteVanish} above.  
\begin{itemize}
\item
If point 1 in $M$ (respectively $R$) has escaped to infinity, then $\phi_{12}$ (respectively $\phi_{13}$) maps such a configuration to $C_+$ (and in general may hit any point in $C_+$, considering that point 2 in $M$ may have also escaped to infinity).  From the definition of the $\Phi_D$ from Lemma \ref{MapsPhi}, we see that these maps hit the first $S^2$ factor.  So the image of such faces under $\Phi_D$ is $C_+ \x S^2 \x S^2$.
\item
If the point on the second strand in $L$ (respectively $R$), escapes to infinity, then $\phi_{31}$ (respectively $\phi_{24}$) takes this configuration into $C_-$ (respectively $C_+$).  Clearly this upper bound on the image is sharp.  So the image of such faces under $\Phi_D$ is $S^2 \x C \x S^2$.
\item
We apply similar reasoning to a face where the point on the third strand in $L$ or $M$ goes to infinity.  We consider the maps $\phi_{42}$ and $\phi_{43}$, and we see that the image of such faces under $\Phi_D$ is $S^2 \x S^2 \x C_-$. 
\item
It is not hard to see that for faces where a point on a strand with two vertices goes to infinity, the images under $\Phi_D$ are contained in the sets mentioned above.
\item
The same applies to faces where only interval vertices on $T$ go to infinity.
\end{itemize}

Now it is easy to see that the image of any face of type (b) in Lemma \ref{InfiniteVanish} is contained in $(H_+ \x H_+ \x H_+) \cup (H_- \x H_- \x H_-)$.  In fact, each interval vertex has either escaped to infinity, in which case it lies in the $xy$-plane, or it has not, in which case it appears to be at the origin from the point of view of the free vertex.


\end{proof}

\begin{remark}
\label{DegenerateLocusBetter}
We can give a more complete description of the images of the type (b) infinite faces.
We reason along similar lines to the last sentence of the proof above, and we consider further the particular linear embeddings with which a string link agrees towards infinity.  From this, we see that their images are precisely the image of the set 
\begin{align*}
\left\{
\begin{array}{c}
(v_1, v_2, v_3) = ((x+\frac{1}{2}s,y - s,z), (x, y -t, z), (x - \frac{1}{2}u, y-u, z))) : \\
x,y,z,s,t,u \in \R \mbox{ and } v_1, v_2, v_3 \mbox{ are all nonzero} 
\end{array}
\right\} 
\subset (\R^3)^3
\end{align*}
under the unit vector map $(\R^3)^3 \to (S^2)^3$ given by $(v_1, v_2, v_3) \mapsto (v_1/|v_1|, v_2/|v_2|, v_3/|v_3|)$.  However, from our perspective on the triple linking number as a degree, this description is not as relevant as the statements in Proposition \ref{DegenerateLocus} above. 
\end{remark}

\subsection{Choosing the spherical forms}
\label{ChoiceOfForms}

We now fix forms $\omega_1, \omega_2, \omega_3$ on the three $S^2$ factors in $S^2 \x S^2 \x S^2$; these are the forms that are pulled back to the configuration bundle to obtain the various $\theta_{ij}$.
The $\omega_i$ will be cohomologous to unit volume forms.  For our topological construction, we need them to have support outside of 
the degenerate locus $\DD$.  
For checking that we get the correct invariant in Section \ref{Indeterminacy} below, it is convenient to take each $\omega_i$ to be supported on a small neighborhood $D_i$ of a point 
$v_i$ 
in $S^2$.  

Recall the larger subspace $\widetilde{\DD}$ from Proposition \ref{DegenerateLocus} above.  Its complement has two components, distinguished by a choice of hemisphere (upper or lower) in the second $S^2$ factor.  Of course these components are contained in the components complementary to $\DD$.  A priori we may choose the $\omega_1 \x \omega_2 \x \omega_3$ to be supported in either component.  However, we will see in the next Subsection that, to get the correct invariant on the nose, we need to choose one component and not the other.  The reason is as follows.

If the image of the infinite faces under $\Phi_D$ had codimension $\geq2$ (as always happens in the ``degeneracy" arguments for closed knots and links), Stokes' Theorem would imply that any choice of form gives the same invariant (see D Thurston's AB Thesis \cite[Section 4.2]{Thurston}).  
In the setting of 
string links, the image of the infinite faces has codimension 1, so our choice of form might affect the resulting invariant (even though any two choices differ by a lower-order invariant, ie, a type--1 invariant in this example).
Keeping in mind that we will validate these choices below, we choose $D_1$ to be a neighborhood of the north pole $v_1:=(0,0,1)$ and $D_2$ and $D_3$ to be neighborhoods of the south pole $v_2:=v_3:=(0,0,-1)$.

\subsection{Eliminating the indeterminacy}
\label{Indeterminacy}

So far, Corollary \ref{MellorCor} only guarantees that $\mu_{123}=I_L - I_M +I_R - I_T$ agrees with the triple linking number up to a type--1 invariant.  
(This is because configuration space integrals only provide an inverse to a map from the \emph{quotient} $\V_m / \V_{m-1}$ to $\W_m$.)
This is the last result of this Section:

\begin{proposition}
\label{OnTheNose}
With our choice of forms in the previous Subsection, $\mu_{123}$ is exactly the triple linking number, ie, we can remove the ``up to a type--1 invariant" indeterminacy from Corollary \ref{MellorCor}.
\end{proposition}

\begin{proof}
The proof of the this Proposition is similar to the proof of  Theorem 5.8 in our paper \cite{MunsonVolicHtpyLinks}, which uses the idea of ``tinker-toy diagrams" of D Thurston \cite{Thurston}.  The ``type--1 invariant" that appears in the indeterminacy must be the sum of a constant (a type--0 invariant) and pairwise linking numbers (the only type--1 invariants).  To check that the constant term is zero, we just have to check that $\mu_{123}$ vanishes on the unlink.  To check that the coefficients of the pairwise linking numbers are zero, it suffices to check the vanishing of $\mu_{123}$ on links $L_{12}, L_{13}, L_{23}$ that have exactly one non-vanishing pairwise linking number and vanishing triple linking number.  

Specifically, we take $L_{ij}$ to be the link where, as we move in the positive $y$-direction, strand $j$ passes under strand $i$ and then over strand $i$, while the other two strands are in the same position as in the unlink.  (Here ``under/over" means in the negative/positive $z$-direction.)  In the case of $L_{13}$, we take strand 3 to pass over strand 2 both before and after linking with strand 1.  It is clear that the $L_{ij}$ each have exactly one nonvanishing pairwise linking number.  The vanishing of the triple linking number can be deduced in a straightforward manner from the definition in Section \ref{MilnorInvariants}.  The least obvious case is $L_{13}$, where we \emph{need} to use the basepoint (0,0,1) in the link complement (chosen in Section \ref{MilnorInvariants}).

We first check that for $D\in \{L,M,R\}$, the integral $I_D$ vanishes on all four of these links.  For any such $D$, the integral $I_D$ counts configurations of the two pairs of points (joined by a chord) such that each pair's difference unit vector is equal to the appropriate $v_i$ and such that the points are on the strands in the order specified by $D$.  We now exploit the fact that $v_i$ has been chosen to be the north or south pole, which makes it clear that the $I_D$ vanish on the unlink.  

Next, in each of $L,M,R$, every strand is touched by at least one chord $c$. 
In either $L_{12}$ or $L_{23}$, there is a strand $i$ (3 or 1, respectively) such that the vertical plane through strand $i$ crosses no other strand.   Thus the pullback of the form corresponding to the chord $c$ is zero, and hence $I_L, I_M, I_R$ all vanish on these links.  The integrals $I_L$ and $I_M$ vanish on $L_{13}$ because the pullback of the form corresponding to a chord between strands 1 and 2 must be zero.  The integral $I_R$ vanishes on $L_{13}$ because strand 3 is never below strand 2, and we chose $\omega_2$ to be supported near the north pole $(0,0,1)$.  The reader who thinks through this calculation carefully will see that, as promised, the choice of form matters!  In fact, it is precisely the choice of the support $\omega_2$ that matters, and we saw that the components complementary to $\widetilde{\DD}$ are distinguished by the hemispheres of the \emph{second} $S^2$ factor.  (On a related note, this shows that the components complementary to $\widetilde{\DD}$ are not connected in the complement of $\DD$.)


For the diagram $T$, $I_T$ counts the number of configurations of three interval vertices on distinct strands plus a free vertex anywhere in $\R^3$ such that each interval vertex is (almost) directly above or below the free vertex.  (Again, recall the choice of $\omega_i$.)  It is not hard to see that none of the four links above has such a configuration.  Hence $I_T$ vanishes on all of them.

\end{proof}

Now we can update Corollary \ref{MellorCor}, eliminating the indeterminacy and rewriting the integrals as we did just before Lemma \ref{MapsPhi}.

\begin{corollary}
\label{MellorCorBetter}
Given our choice of basepoint in the string link complement and choices of spherical forms $\omega_i$, the triple linking number for string links is the sum of integrals $\mu_{123} =  I_L - I_M + I_R - I_T $ where 
\begin{align*}
I_L = \int_{F_L \x S^2} \omega_1 \theta_{31}\theta_{42}, & &
I_M= \int_{F_M \x S^2} \theta_{12}\omega_2(-\theta_{43}), & &
I_R=\int_{F_R} \theta_{13}\theta_{24} \omega_3, & &
I_T= \int_{F_T} \theta_{14}\theta_{24}\theta_{34}.
\end{align*}
\end{corollary}

\section{Gluing manifolds with faces}
\label{GluingMfds}

Lemmas \ref{PrincipalCancel} and \ref{InfiniteVanish} show that we should glue our bundles along their principal faces and work relative to, or collapse, the infinite faces.  To prove Theorem \ref{DegreeThm}, we need the glued space to be a manifold with corners.  To embed it with a tubular neighborhood and normal bundle and thus prove Theorem \ref{PTThm}, we need it to be a nice type of manifold with corners, namely a manifold with faces.  We first recall some definitions related to manifolds with faces.  Next we prove some general statements which will be useful in understanding how to do a Pontrjagin--Thom construction with our glued space.  Then we construct a glued space $E_g$, and, as the main result of this Section, we show that $E_g$ is a manifold with faces.

\subsection{Gluing manifolds with faces}
\label{GluingMfdsWFaces}
We start by reviewing some basic definitions.  The notion of manifolds with faces and $\langle N \rangle$--manifolds goes back to work of J\"{a}nich \cite{Janich}.  Here, as in our previous work \cite{Rbo}, we also follow the work of Laures \cite{Laures}.  First recall that in a \emph{manifold with corners} $X$ every $x\in X$ has a neighborhood diffeomorphic to a neighborhood of the origin in $\R^{n-k} \x [0,\infty)^k$.  Call the number $k$ the \emph{codimension} $c(x)$ of the point $x$.  Let $C_k$ be the set of points in $X$ of codimension $k$.  Call the closure of a connected component of $C_1$ a \emph{connected face}.  We say that $X$ is a \emph{manifold with faces} if $X$ is a manifold with corners such that for any $x\in X $, $x$ is contained in $c(x)$ different connected faces of $X$.  A \emph{$\langle N \rangle$--manifold} is a manifold with faces $X$ together with a decomposition of the boundary $\d X = \d_1X \cup \dotsb \cup \d_N X$ such that each $\d_i X$ is a disjoint union of connected faces and such that $\d_i X \cap \d_j X$ is a disjoint union of connected faces of each of $\d_i X$ and $\d_j X$.

It will be convenient to know the following basic lemma.  Its proof is not too difficult, so we leave it as an exercise for the reader (cf.~Joyce's paper \cite[Remark 2.11]{Joyce}).

\begin{lemma}
\label{MfdsWFacesRAngleMfds}
Any compact manifold with faces $X$ can be given the structure of a $\langle N \rangle$--manifold for some $N$.
\qed
\end{lemma}

We continue discussing glued spaces in general terms.

\begin{lemma}
\label{GluedAreMfdsWFaces}
Let $X_1$ and $X_2$ be two manifolds with faces.  Suppose $\SS_i \subset \d X_i$ for $i=1,2$ are codimension--1 faces of $X_i$ and are diffeomorphic as manifolds with corners via an orientation-reversing diffeomorphism $g\co \SS_1\cong \SS_2$.  Then we can form a new manifold with faces $X=X_1 \cup_g X_2$ by gluing the $X_i$ via $g$.  
\end{lemma}
\begin{proof}
We first check that $X$ is a manifold with corners.  Let $x\in \SS_i$ be a point with codimension $c(x)=k$. Then in each $X_i$ (using the diffeomorphism $g$), $x$ has a neighborhood diffeomorphic to a ``ball" around the origin in $\R^{n-k}\x[0,\infty)^k$.  If we glue two such balls along the codimension--1 faces which map to the $\SS_i$, then the result can obviously be identified with a ``ball" in $\R^{n-k+1} \x [0,\infty)^{k-1}$.  We use this to give $X$ the structure of a smooth manifold with corners.  Note that a point in $\SS_i$ that was codimension $k$ in $X_i$ now has codimension $k-1$ in $X$.

To give $X$ the structure of a manifold with faces, we glue all the strata which intersect $\SS_i$ in pairs (which can be thought of as ``mirror images" across $\SS_i$).  Thus the (codimension--1) connected faces of $X$ consist of those connected faces of $X_i$ which don't intersect $\SS_i$, together with one connected face for every element in the set
\[
\{(V_1, V_2) \>\>|\>\> V_i \mbox{ a (codim.--1) conn. face of } X_i, \>\> V_i \cap \SS_i\neq \emptyset, \>\> V_i\neq \SS_i, \>\>g(V_1 \cap \SS_1) = V_2 \cap \SS_2\}.
\]
Suppose a point $x\in \SS_i$ had codimension $k$ in (either) $X_i$.  Then it was contained in $k$ faces of $X_i$.  In $X$, it is contained in $k-1$ faces: one for each of the $k$ faces except $\SS_i$ itself.  But as noted above, $x$ has codimension $k-1$ in $X$.  The condition is obviously satisfied for all $x \notin \SS_i$.
\end{proof}

By Lemma \ref{MfdsWFacesRAngleMfds}, the glued-up space $X$ is a $\langle N \rangle$--manifold for some $N$.  The content of Laures' \cite[Proposition 2.1.7]{Laures} is that compact $\langle N \rangle$--manifolds are precisely the compact spaces which admit \emph{neat} embeddings into Euclidean space with corners $\R^M \x[0,\infty)^\N$.   
\begin{corollary}
\label{GluingCor}
Let $X_1$ and $X_2$ be two \emph{compact} manifolds with faces.  Suppose $\SS_i \subset \d X_i$ and $g\co \SS_1\cong \SS_2$ are as above.  Then $X=X_1 \cup_g X_2$ can be neatly embedded into some Euclidean space with corners.
\qed
\end{corollary}
Roughly speaking, a neat embedding is one which respects the corner structure.  For a precise definition, see Laures' paper \cite[Definition 2.1.4]{Laures} or our paper \cite[Definition 3.1.2]{Rbo}.  For our construction, the important features are  (1) that neat embeddings have well defined normal bundles and (2) that the manifold being embedded meets all the boundary strata of $\R^M \x[0,\infty)^\N$ perpendicularly.

Finally note how the construction in Lemma \ref{GluedAreMfdsWFaces} can be repeated inductively.  That is, suppose $X_1,...,X_k$ are manifolds with faces, and $\SS_1, \SS'_1, ... , \SS_m, \SS'_m$ are codimension--1 faces in $X_{i_1}, X_{j_1}, ... , X_{i_m}, X_{j_m}$ such that all the $\SS_\ell$ and $\SS'_\ell$ are pairwise disjoint and such that there are diffeomorphisms $g_{\ell} : \SS_{\ell} \cong \SS'_{\ell}$ for each $\ell$.  Then we can form a glued space $X_1 \sqcup \dotsb \sqcup X_k/ \sim$, where the quotient is determined by $g_1,...,g_m$.  From the proof of Lemma \ref{GluedAreMfdsWFaces}, we see that this glued space is again a manifold with faces, hence an $\langle N \rangle$--manifold for some $N$.

\subsection{Fiberwise gluings of manifolds with faces}
\label{FibwiseGluingMfdsWFaces}
Now let $\xymatrix{F_1\ar[r] &  E_1\ar[r] & B}$ and $\xymatrix{F_2\ar[r] &  E_2\ar[r] & B}$ be fiber bundles where each $F_i$ is a compact manifold with faces.  Let $F_i(b)$ denote the fiber of $E_i$ over $b\in B$. Suppose that (for $i=1,2$)  we have sub-bundles $\SS_i \subset E_i \to B$ where the fiber $F(\SS_i)(b)$ over each $b\in B$ is a codimension--1 face of $F_i(b)$.  Suppose that for each $b\in B$, $F(\SS_1)(b), F(\SS_2)(b)$ are diffeomorphic as manifolds with corners via an orientation-reversing diffeomorphism $g_b:F(\SS_1)(b)\cong F(\SS_2)(b)$, and suppose further that these diffeomorphisms vary smoothly over $B$.  (That is, fix one copy $F(\SS)$ of $F(\SS_i)(b)$; then given a trivialization of $E_i$ over $U\subset B$ as $U\x F_i$, the induced map from $U\x F(\SS_i)$ to this fixed $F(\SS)$ is smooth.)  Then we can fiberwise glue $E_1$ to $E_2$ along $F(\SS_i)$ to form a bundle $\xymatrix{F_g\ar[r] &  E_g \ar[r] & B}$ whose fibers $F_g = F_1 \cup_g F_2$ are manifolds with faces, or $\langle N \rangle$--manifolds for some $N$.  Similarly, we can carry out a fiberwise analogue of the construction at the end of the previous paragraph with more than two bundles $E_1,...,E_k$, provided that the faces along which we glue are pairwise disjoint.

\subsubsection{Fiberwise integration on glued manifolds with corners}
Suppose that, as above, we glue bundles  $\xymatrix{F_i\ar[r] &  E_i\ar[r] & B}$ ($i=1,2$) to form a bundle  $\xymatrix{F_g\ar[r] &  E_g \ar[r] & B}$.  Suppose $\alpha_i$ for $i=1,2$ are forms on $E_i$ such that the forms $g^*(\alpha_2)$ and $\alpha_1$ agree on $\SS_1$.
The forms $\alpha_i$ can then be glued to a form $\alpha$ on $E_g$ which is at least piecewise smooth.  Without any further assumptions, we do not have fiber integration on the level of cohomology for the bundles $E_i\to B$.  But on the level of forms, this much is obvious:

\begin{lemma} 
\label{SumIntegralsIsIntegralGlued}
In the above situation,
\[
\int_{F_1} \alpha_1 + \int_{F_2} \alpha_2 = \int_{F_1 \cup_g F_2} \alpha.
\]
A similar statement holds in the case of gluing bundles $E_1,...,E_k$ with $k>2$.
\qed
\end{lemma}

Since integration can be done piecewise, the manifold with faces structure (or even the manifold with corners structure) on $E_g$ is not strictly necessary for this lemma.  But the particular glued space $E_g$ which we will construct will anyway be a manifold with faces.

\subsection{Gluing the bundles over the link space}
\label{GlueBundlesLinkSpace}
Now we construct our particular glued space $E_g$, which is a bundle over $\L_3$.

The spaces $E_D$ ($D=T,L,M,R$) have precisely the corner structure of their fibers and hence principal faces corresponding to those in the fibers.  Let $\SS_1,\SS_2,\SS_3$ denote the principal faces of $E_T=E[1,1,1;1]$, and let $\SS_L, \SS_M, \SS_R$ denote the principal faces of $E_L=E[2,1,1;0], \>E_M=E[1,2,1;0], \>E_R=E[1,1,2;0]$ respectively.  We denote the fibers of these faces (or faces of the fibers) $F(\SS_i)$, $i=1,2,3$ and $F(\SS_D)$, $D=L,M,R$.

We observed in Section \ref{CptfdConfigSpaces} that each of $F(\SS_1), F(\SS_2), F(\SS_3)$ is diffeomorphic to $F[1,1,1;0] \x S^2$, and that each of $F(\SS_L), F(\SS_M), F(\SS_R)$ is diffeomorphic to $F[1,1,1;0]$.  Moreover, we have an orientation-preserving diffeomorphism from $F(\SS_i)$ ($i=1,2,3$) to $F(\SS_D) \x S^2$ ($D=L,M,R$) where the orientation on $S^2$ comes from the one we fixed on $\R^3$.

We glue $F_L$, $F_M$, and $F_R$ to $F_T$ by first reversing the orientation on $F_T$, and then identifying $F(\SS_L) \x S^2$ with $F(\SS_1)$, $F(\SS_M) \x S^2$ with $F(\SS_2)$, and $F(\SS_R) \x S^2$ with $F(\SS_3)$.  
Note that we have a specific (orientation-reversing) diffeomorphism over each fiber, and these diffeomorphisms vary smoothly over $\L_3$.  Thus the gluing of the fibers specifies a gluing of the total spaces, ie, a gluing of $E_L \x S^2, E_M \x S^2, E_R \x S^2$ to $E_T$.  Denote the result $E_g$.  This space is a bundle over $\L_3$ with fiber denoted $F_g$.  

\begin{figure}[h]
\includegraphics[height=3pc]{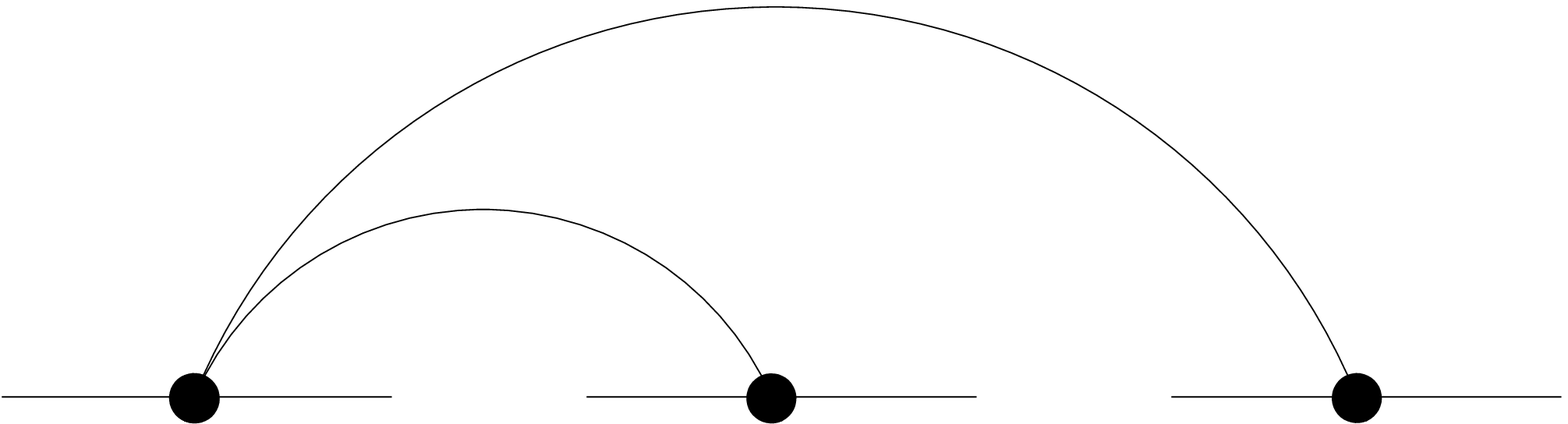},
\includegraphics[height=2pc]{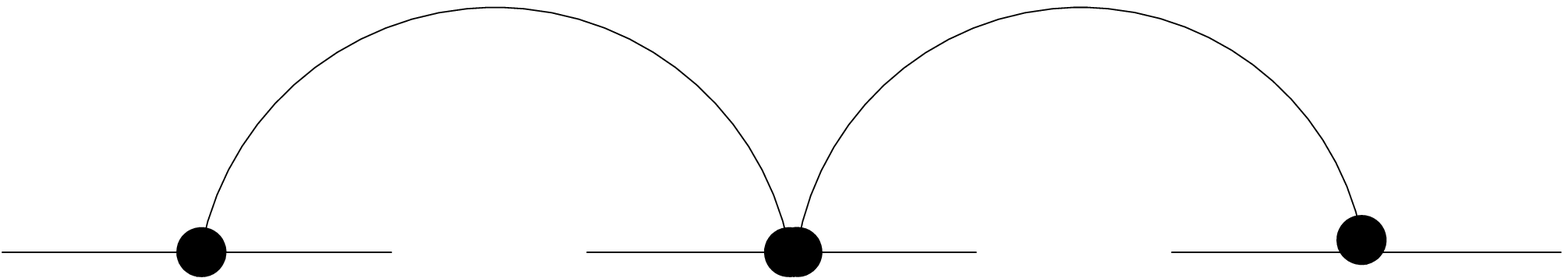},
\includegraphics[height=3pc]{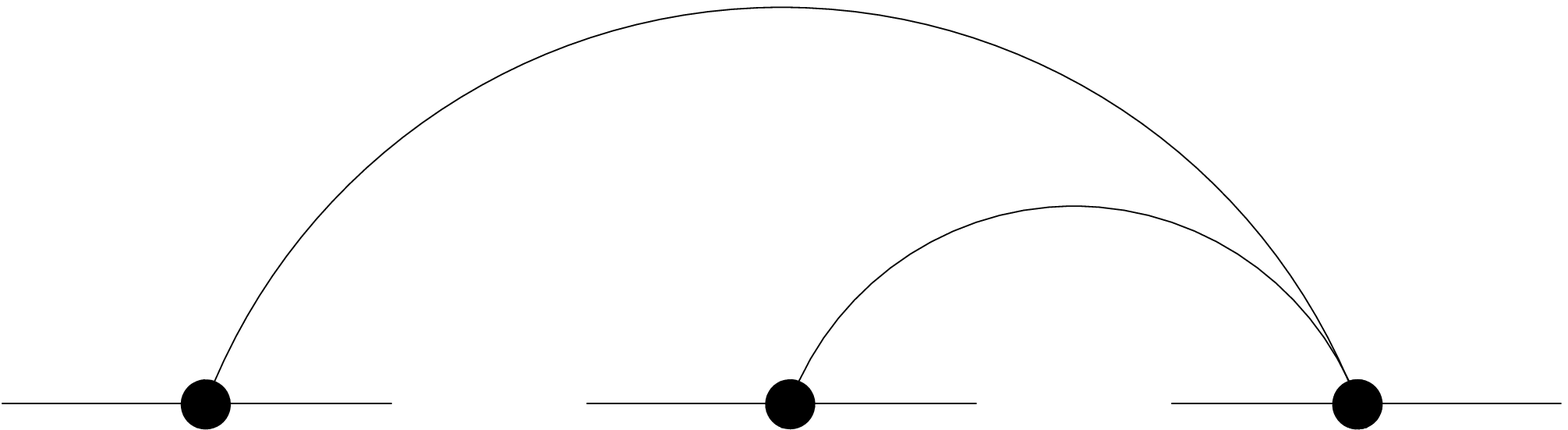}
\caption{The diagrams $L'', M'', R''$ corresponding to the three codimension--1 faces along which we glue.}
\label{CollidedDiagrams}
\end{figure}

\begin{lemma}
\label{PrincipalFacesDisjoint}
The three principal faces $\SS_1, \SS_2, \SS_3$ of $E_T$ are pairwise disjoint.  
\end{lemma}
\begin{proof}
Recall that a face in the Axelrod--Singer/Fulton--MacPherson compactifications $C_n[M]$ is indexed by a set $\{S_1,...,S_k\}$ of subsets of $\{1,...,n\}$ which are nested or disjoint.  The intersection of two faces indexed by $\{S_1,...,S_k\}$ and $\{S'_1,...,S'_\ell\}$ is the face indexed by $\{S_1,...,S_k, S_1',...,S'_\ell\}$, if these sets are nested or disjoint.  But in our case, the three faces are indexed by $\{\{1,4\}\}, \{\{2,4\}\}, \{\{3,4\}\}$, and each pair of sets is neither nested nor disjoint.  Hence no pair of $\SS_1, \SS_2, \SS_3$ intersect.  
\end{proof}

Now as indicated at the end of Sections \ref{GluingMfdsWFaces} and \ref{FibwiseGluingMfdsWFaces}, the proof of Lemma \ref{GluedAreMfdsWFaces} makes the main result of this Section clear:

\begin{proposition}
The glued total space $E_g$ (or its fiber $F_g$) is a manifold with faces, hence a $\langle N \rangle$--manifold for some $N$.
\qed
\end{proposition}

Note that this implies of course that $E_g$ is a manifold with corners.

\section{Finishing the proof of Theorem \ref{DegreeThm}}
\label{ProofThm1}
To recap how far along we are in proving Theorem \ref{DegreeThm}, we first showed in Section \ref{MilnorIntegrals} (Corollary \ref{MellorCorBetter}) how a certain sum of configuration space integrals combine to yield the triple linking number. 
We then showed in Section \ref{GluingMfds} that there is a bundle $E_g$ over the space of links whose fibers are manifolds with faces.  
In this Section, we prove Theorem \ref{DegreeThm} by topologically reinterpreting the configuration space integral expression from Section \ref{MilnorIntegrals} in terms of the glued bundle $E_g$ (and its fiber $F_g$) from Section \ref{GluingMfds}.

Recall from Lemma \ref{MapsPhi} the maps 
\begin{align}
\label{MapsToSpheres}
\begin{array}{c}
\Phi_L = (\mathrm{proj}_{S^2}, \phi_{31}, \phi_{42}) \co E_L \x S^2 \to S^2 \x S^2 \x S^2 \\
\Phi_M = (\phi_{12}, \mathrm{proj}_{S^2}, \phi_{43}) \co E_M \x S^2 \to S^2 \x S^2 \x S^2 \\
\Phi_R = (\phi_{13}, \phi_{24}, \mathrm{proj}_{S^2}) \co E_R \x S^2 \to S^2 \x S^2 \x S^2 \\
\Phi_T = (\phi_{14}, \phi_{24}, \phi_{34}) \co E_T \to S^2 \x S^2 \x S^2.
\end{array}
\end{align}
whose domains are the four spaces which we glued together.
We fixed 2--forms $\omega_1, \omega_2, \omega_3 \in \Omega^2 (S^2)$ in Section \ref{ChoiceOfForms}, and we take each $\theta_{ij}$ to be the pullback via $\phi_{ij}$ (as above) of the appropriate $\omega_i$.  (Luckily the only $\phi_{ij}$ with the same indices map to the same $S^2$ factor, so there is no abuse of notation.)

Note that the forms 
\begin{align*}
\theta_{14}\theta_{24}\theta_{34}\in \Omega^6(E_T)
 & & \mbox{ and } & & 
\omega_1 \theta_{31}\theta_{42}
\in \Omega^6(E_L \x S^2)  
\end{align*}
agree on the common face
\begin{align*}
E_T \supset \SS_L \x S^2 \>\>\>\>\> \cong \>\>\>\>\>  \SS_1 \subset  E_L \x S^2.\>\>\>\>\>\>\>\>
\end{align*}

Similarly, the forms  
\begin{align*}
\theta_{14}\theta_{24}\theta_{34} \in \Omega^6(E_T) 
& & \mbox{ and } & & 
\theta_{12} \omega_2 \theta_{43}
\in\Omega^6(E_M \x S^2) 
\end{align*}
agree on the common face
\[
E_T \supset \SS_M \x S^2 \>\>\>\>\> \cong \>\>\>\>\> \SS_2 \subset  E_M \x S^2,\>\>\>\>\>\>\>\>
\]
and 
\begin{align*}
\theta_{14}\theta_{24}\theta_{34} \in \Omega^6 E_T
& & \mbox{ and } & & 
\theta_{13}\theta_{24}\omega_3 \in\Omega^6(E_R \x S^2) 
\end{align*}
agree on the common face
\[
E_T  \supset \SS_R \x S^2 \>\>\>\>\> \cong \>\>\>\>\>  \SS_3 \subset E_R \x S^2.\>\>\>\>\>\>\>\>
\]
Let $\beta$ be the 6--form on $E_g$ obtained by gluing together of these four 6--forms.

\begin{remark}
\label{OrderOfIndices}
In choosing the order of indices in the $\theta_{ij}$ above, we were guided by the diagrams: we fix the form in $\Omega^6(E_T)$, and we orient the edge in the diagram $T$ corresponding to $\theta_{ij}$ from $i$ to $j$; then we imagine, for example, points 1 and 4 colliding to yield the diagram $L''$ in Figure \ref{CollidedDiagrams}, and then moving apart along the link to yield the diagram $L$.  The choice of form on $E_L \x S^2$ came from the fact that the edges in $L$ are now oriented from point 3 to point 1 and from point 4 to point 2.  
\end{remark}

\begin{proposition}
\label{Mu123IsIntegralOnGluedSpace}
$\mu_{123}=\int_{F_g} \beta$.
\end{proposition}

\begin{proof}
From Corollary \ref{MellorCorBetter}, 
\begin{align*}
\mu_{123}
&= \int_{F[2,1,1;0]\x S^2} \omega_1 \theta_{31}\theta_{42} - \int_{F[1,2,1;0] \x S^2} \theta_{12} \omega_2 (-\theta_{43}) +
\int_{F[1,1,2;0] \x S^2} \theta_{13}\theta_{24} \omega_3  
- \int_{F[1,1,1;1]} \theta_{14}\theta_{24}\theta_{34} \\
&= \int_{F[2,1,1;0]\x S^2} \omega_1 \theta_{31}\theta_{42} + \int_{F[1,2,1;0] \x S^2} \theta_{12} \omega_2 \theta_{43} +
\int_{F[1,1,2;0] \x S^2} \theta_{13}\theta_{24} \omega_3 + 
\int_{\overline{F[1,1,1;1]}} \theta_{14}\theta_{24}\theta_{34} 
\end{align*}
where $\overline{F[1,1,1;1]}$ is $F[1,1,1;1]=F_T$ with its orientation reversed.
By Lemma \ref{SumIntegralsIsIntegralGlued} and the fact that the forms agree on the appropriate pairs of faces, this last expression is  $\int_{F_g} \beta$.
\end{proof}

Recall that we chose $\omega_i$ 
so that the restriction of $\Phi_D^*\alpha:=\Phi_D^* (\omega_1 \x \omega_2 \x \omega_3)$ to any infinite face (for any $D$) vanishes.  
This vanishing on infinite faces means that $\beta$ represents a class in $H^6(E_g, \d E_g)$.  If we consider the integration over just one fixed link in the base $\L_3$, then we can think of $\beta$ as a top-dimensional class in $H^6(F_g, \d F_g)$.  We have seen that $F_g $ is a manifold with corners, and $\d F_g$ is precisely its boundary, so there exists a fundamental class $[F_g, \d F_g]$.  In such a situation, integration is the same as pairing the cohomology class $[\beta]$ with the fundamental class $[F_g, \d F_g]$:
\[
\mu_{123} = \int_{F_g} \beta = \langle [\beta], [F_g, \d F_g] \rangle
\]
Since the $\omega_i$ represent integral cohomology classes, so does $\beta$, so the right-hand side can be viewed in integral (co)homology.

The maps (\ref{MapsToSpheres}) descend to a map $E_g \to S^2 \x S^2 \x S^2$ and, looking over just one link, to a map  
\begin{align}
\label{FiberMapToSpheres}
\xymatrix{ F_g \ar[r] & S^2 \x S^2 \x S^2}
\end{align}
with 
\[
\xymatrix{& \>\>\>\>\>\>\>\> H^6(F_g; \Z) \ni [\beta] & [\omega_1 \x \omega_2 \x \omega_3] \in H^6(S^2\x S^2 \x S^2; \Z) \ar@{|->}[l]}
\]
where $[\beta]$ is really the image of $[\beta]$ under $H^6(F_g, \d F_g) \to H^6(F_g)$, and where $[\omega_1 \x \omega_2 \x \omega_3]$ is a generator, since the $\omega_i$ represent generators of $H^2(S^2; \Z)$.  

Recall the notation $\DD\subset S^2 \x S^2 \x S^2$ for the image of $\d F_g$, which we described in Proposition \ref{DegenerateLocus} and Remark \ref{DegenerateLocusBetter}.
We can rewrite the map (\ref{FiberMapToSpheres}) above as
\begin{align*}
\xymatrix{ 
F_g / \d F_g \ar[r] & S^2 \x S^2 \x S^2 / \DD.
}
\end{align*}

Recall that $F_g = F_T \, \cup \, F_L \x S^2 \, \cup \, F_M\x S^2 \, \cup \, F_R \x S^2$.  Recall also that the compactifications are homotopy-equivalent to their interiors and have collar neighborhoods of their boundaries.  We now ignore the corner structure since we are purely in the realm of topology.  Up to homeomorphism $F_T \cong I^3 \x (D^3 \setminus \{\mbox{3 $D^3$'s}\})$; the three missing $D^3$'s correspond to the principal faces $F(\SS_1),F(\SS_2), F(\SS_3)$, which are homeomorphic to $I^3 \x S^2$.  The other three $F_D\x S^2$ are homeomorphic to $I^4 \x S^2 = I^3 \x (I \x S^2)$, and the principal faces $F(\SS_D)$ correspond to the $I^3 \x S^2$ pieces of the boundary obtained by choosing one endpoint in the $I$ factor.  Thus $F_g \cong I^3 \x (D^3 \setminus \{\mbox{3 $D^3$'s}\})$, since gluing $F_L,F_M,F_R$ to $F_T$ corresponds to filling in a neighborhood of three boundary components in the $(D^3 \setminus \{\mbox{3 $D^3$'s}\})$ factor.

We extend the above map to the left, where $CX$ denotes the cone on $X$:
\[
 \xymatrix{ D^6 \cong I^3 \x \left((D^3 \setminus \{\mbox{3 $D^3$'s}\}) \cup \bigcup_1^3 CS^2 \right)  \ar[r] & D^6/ \d D^6 
  \ar[r] & F_g / \d F_g  \ar[r] &  S^2 \x S^2 \x S^2 / \DD}
\]
The first map is just the quotient by the boundary.  The middle map is the quotient by the image of $I^3 \x \{3 \mbox{ cone points}\}$ under the first map.  This image in $D^6/ \d D^6$ is $\bigvee_1^3 S^3$, so this middle map is the quotient 
\[
\xymatrix{\bigvee_1^3 S^3 \ar@{^(->}[r] & S^6 \ar@{->>}[r] & F_g/\d F_g}
\]
which induces  isomorphisms in $H^6$ and $H_6$.  Under the isomorphism in $H^6$, $[F_g, \d F_g]$ corresponds to a fundamental class $[S^6]$.  Thus the triple linking number $\mu_{123}= \langle [\beta], [F_g, \d F_g] \rangle$ is given by the pairing of the fundamental class $[S^6]$ with the pullback of the cohomology class $[\alpha]=[\omega_1 \x \omega_2 \x \omega_3]$ via the left-hand map below; the class $[\alpha]$ furthermore maps to a generator of $H^6(S^2 \x S^2 \x S^2; \Z)$ via the right-hand map:
\[
\xymatrix{ S^6 \ar[r] & S^2 \x S^2 \x S^2 / \DD & S^2 \x S^2 \x S^2 \ar@{->>}[l]}
\]
This proves Theorem 1. 
\qed

The constraints we have imposed on the $\omega_i$ imply that $\alpha = \omega_1 \x \omega_2 \x \omega_3$ has support in a 
6-ball contained in 
$S^2 \x S^2 \x S^2 \setminus \DD$ 
Quotienting by the complement of the support of $\alpha$ thus gives a space $S(\alpha)$ which has a fundamental class.  It is not hard to see that the degree of the composition 
\[
\xymatrix{ S^6 \ar[r] & S^2 \x S^2 \x S^2 / \DD  \ar@{->>}[r] & S(\alpha)}
\]
is the triple linking number $\mu_{123}$.  So we have exhibited it as a degree, as promised in the Introduction.

\section{Integration along the fiber via Pontrjagin--Thom}
\label{IntegrationViaPT}

This Section reviews a general, classical result relating integration along the fiber to a Pontrjagin--Thom construction.  
We begin with the case where the fiber has no boundary, and then treat the case of a fiber with boundary.  The latter case proceeds very similarly to the former, but we start with the boundaryless case for ease of readability.

\subsection{Warmup: the case of boundaryless fibers}

Suppose $F\to E \to B$ is a fiber bundle of compact manifolds (without boundary).  Then integration along the fiber is a chain map (see equation (\ref{Stokes}) as well as Bott and Tu's book \cite[page 62]{BottTu}) and hence induces a map in cohomology 
\[
\mbox{$\int_F$}\co \xymatrix{ H^p(E) \ar[r] &  H^{p-k}(B).}
\]
Since we will use de Rham cohomology (or compare other constructions to those in de Rham cohomology) all (co)homology will be taken with real coefficients for the rest of this Section.  

The above map can also be produced as follows.  Since $E$ is a compact manifold, we can embed it in some $\R^N$ and use this to embed the bundle into a trivial bundle:
\[
\xymatrix{
E \ar@{^(->}[rr]^e \ar[dr]& & \R^N \x B \ar[dl] \\
 & B & 
}
\]
Then $E$ has a tubular neighborhood $\eta(E)$ in $\R^N \x B$ which is diffeomorphic to the normal bundle $\nu = \nu(e)$ of the embedding $e$.  
Taking the quotient by the complement gives a Pontrjagin--Thom collapse map $\tau$:
\[
\xymatrix{
\R^N \x B \ar[rr] \ar[dr] &  & \R^N \x B / (\R^N \x B - \eta(E)) \cong E^\nu \\
 & \Sigma^N B_+ \ar[ur]_\tau &
}
\]
Here $E^\nu$ denotes the Thom space $D(\nu)/S(\nu)$ of the normal bundle $\nu \to E$ and $\Sigma^N B_+$ is the $N$--fold suspension of $B_+ = B \sqcup \{*\}$.  The map above factors through $\Sigma^N B_+$ because that space is just the one-point compactification of $B\x \R^N$, and because $E$ is compact.  We consider the map in cohomology induced by  $\tau\co \Sigma^N B_+ \to E^\nu$. \\

Lemma \ref{IntegrationIsPTEasy} below is the restatement of Lemma \ref{IntegrationIsPTBoundary} (whose statement and proof appears in the next Subsection) in the special case where the boundary is empty.  The reader may prove it as an exercise, or see how the proof of Lemma \ref{IntegrationIsPTBoundary} simplifies when the boundary is empty, or just deduce Lemma \ref{IntegrationIsPTEasy} from the statement of Lemma \ref{IntegrationIsPTBoundary}.

\begin{lemma}  
\label{IntegrationIsPTEasy}
Given a bundle $\xymatrix{F \ar[r] &  E \ar[r]^\pi &  B}$ with all three spaces compact manifolds without boundary and $k=\dim F$, then $\int_F\co H^p(E) \to H^{p-k}(B)$  agrees with the composition 
\[
\xymatrix{ H^{p+k}(E) \ar[rr]^-{\mbox{Thom }\cong}&  & \tH^{p+N} (E^{\nu}) \ar[r]^-{\tau^*} & \tH^{p+N} (\Sigma^N B_+) \ar[r]^-{\Sigma \> \cong} & H^{p}(B) }
\]
induced by the Pontrjagin--Thom collapse map, the Thom isomorphism, and the suspension isomorphism.
\end{lemma}

\subsection{Fibers with boundary}

Now let $F \to E \to B$ be a fiber bundle with $F$ a $k$--dimensional compact manifold with boundary, or even a manifold with faces.
Continue to assume that $B$ is a compact manifold (without boundary).
Considering the boundary of the fiber $\d F \subset F$ gives us a bundle $\d E \to B$ which is a sub-bundle of $E\to B$.   Let $\beta$ be a $(p+k)$--form on $E$.  Suppose the restriction of $\beta$ to $\d E$ is zero.  Then the corresponding cohomology class $\beta \in H^{p+k}(E)$ comes from a class in $H^{p+k}(E, \d E)$.  

Alternatively, $E$ has a collar neighborhood $[0,1) \x \d E$ of its boundary, so $E$ is homotopy-equivalent to the complement of a collar, say $[0,\frac{1}{2})\x \d E$.  Under the homotopy, the form $\beta$ becomes a form on the interior of $E$ with compact support.
Thus we may view $\beta$ as an element in cohomology with compact support, $\beta \in H^{p+k}_c E$.

The vanishing on  $\d E$ of $\beta$ implies by Stokes' Theorem that integration along the fiber produces a well defined cohomology class $\int_F \beta \in H^{p-k}(B)$ (again cf.~equation (\ref{Stokes}) or \cite[page 62]{BottTu}).

Since $F$ is a manifold with faces it can be embedded neatly in some Euclidean space with corners $\R^{\NN, M}:= [0,\infty)^N \x \R^M$ (ie, in a way that respects the corner structure).  Recall that a neat embedding has a well defined normal bundle.  

Since all the corner structure in $E$ comes from that in $F$, we can neatly embed the bundle $E\to B$ into a trivial bundle:
\[
\xymatrix{
E \ar@{^(->}[rr]^e 
\ar[dr]
& & \R^{\NN, M} \x B 
\ar[dl] \\
 & B & 
}
\]
Once again a tubular neighborhood $\eta(E)\subset \R^{\NN, M} \x B$ is diffeomorphic to the normal bundle $\nu = \nu(e)$ of the embedding $e$.  
Taking the quotient by the complement gives a Pontrjagin--Thom collapse map $\tau$:
\[
\xymatrix{
\R^{\NN, M} \x B \ar[rr] \ar[dr] &  & \R^{\NN, M} \x B / (\R^{\NN, M} \x B - \eta(E)) \cong E^\nu \\
 & C^N\Sigma^M B_+ \ar[ur]_\tau &
}
\]
The map above factors through $C^N \Sigma^M B_+$ (where $C^N$ denotes $N$--fold cone) because that is the one-point compactification of $\R^{\NN, M} \x B$.  Restrict the embedding $e$ to $\d E\subset E$, and let $\d E^\nu \subset E^\nu$ be the result of the corresponding restriction of the collapse map.  If we let 
\[
\d(C^N\Sigma^M B_+) = \d ([0, \infty)^N) \x \R^M \x B \subset [0, \infty)^N \x \R^M \x B \subset C^N\Sigma^M B_+
\]
then clearly
\[
e\co \d E \incl \d(C^N\Sigma^M B_+)
\]
so restricting the collapse map above gives a map of pairs
\[
(C^N\Sigma^M B_+, \> \d(C^N\Sigma^M B_+)) \to (E^\nu, \> \d E^\nu).
\]
The resulting map on quotients is  $\tau\co \Sigma^{N+M} B_+ \to E^\nu / \d E^\nu$. \\

\begin{lemma}
\label{IntegrationIsPTBoundary}
Suppose that $F\to E\to B$ is a fiber bundle with $B$ a compact manifold, $F$ a $k$--dimensional compact manifold with boundary.  Let $\beta \in \Omega^{p+k}(E)$ be a form whose restriction to $\d E$ is zero.  Then the class $\int_F \beta$ is precisely the image of $\beta$ under the composition 
\[
\xymatrix{ H^{p+k}(E, \d E) \ar[rr]^-{\mbox{Thom }\cong} & & \tH^{p+N+M} (E^\nu / \d E^\nu) \ar[r]^{\tau^*} & \tH^{p+N+M} (\Sigma^{N+M} B_+) \ar[r]^-{\Sigma \> \cong} & H^{p}(B) }
\]
induced by the Pontrjagin--Thom collapse map together with the Thom isomorphism and suspension isomorphism.
\end{lemma}

\begin{proof}
We have $\beta \in H^{p+k}(E, \d E) \cong H^{p+k}_c(E - \d E)$ and $\int_F \beta \in H^p_c(B)\cong H^p(B)$ since $B$ is compact.  
We first check that integration along the fiber is Poincar\'{e} dual to the map $\pi^*$ induced by $\pi\co E\to B$.  That is, we check that $\int_F \beta \cap [B]= \pi_*(\beta\cap[E])$ for any class $\beta \in H^*(E, \d E)$.

Let $\langle \cdot, \cdot \rangle$ denote the pairing between cohomology and homology.  Let $\alpha \in H^{m-p} B$ where $m=\dim B$.  
Using the relationship between cap and cup products and the ``naturality" property of cap products, we have the following:

\begin{align*}
\langle \alpha,\> \mbox{$\int_F \beta$} \cap [B]\rangle &= \langle \alpha \cup \left(\mbox{$\int_F \beta$} \right),\> [B]\rangle \\
&=\int_B \alpha \wedge \left(\mbox{$\int_F \beta$} \right) \\
&=\int_E \pi^* \alpha \wedge \beta 
\end{align*}
The equality of the last two integrals above follows from a Fubini Theorem
applied to the pairings $\int_B\co H^{m-p}B \otimes H^p_c B \to \R$ and $\int_E\co H^{m-p} (E-\d E) \otimes H^{p+k}_c(E-\d E) \to \R$.  
(See the book by Greub, Halperin, and Vanstone \cite[pages 307--309]{GHV}; see also the paper of Auer \cite{AuerPD} for a characterization of fiber integration in terms of Poincar\'{e} duality.)  
We continue:
\begin{align*}
\langle \alpha,\> \mbox{$\int_F \beta$} \cap [B]\rangle &=\int_E \pi^* \alpha \wedge \beta \\
&=\langle \pi^* \alpha \cup \beta,\> [E, \d E]\rangle \\
&=\langle \pi^*\alpha,\> \beta \cap [E, \d E]\rangle \\
&=\langle \alpha,\> \pi_*(\beta \cap [E, \d E])\rangle
\end{align*}
Since this holds for any $\alpha\in H^{m-p}B$, we must have $\int_F \beta \cap [B]= \pi_*(\beta\cap[E])$.  Thus the following diagram commutes:
\begin{equation}
\label{Integration=PD}
\xymatrix{
H^{p+k}(E) \ar[r]^{- \cap[E, \d E]} \ar[d]_{\int_F} & H_{m-p}(E) \ar[d]_{\pi_*} \\
H^{p}(B) & H_{m-p} (B) \ar[l]_{( - \cap [B])^{-1}}
}
\end{equation}

Next we claim the diagram below commutes:
\[
\xymatrix{
H^{p+k}(E, \d E) \ar[rrr]^{- \cap[E, \d E]}_\cong \ar[d]_{\mbox{relative Thom } \cong} & & & H_{m-p}(E) \ar[d]_{\cong} \\
H^{p+N+M}(D(\nu), S(\nu) \cup D(\nu|_{\d E}))\ar[rrr]^-{- \cap [D(\nu), S(\nu) \cup D(\nu|_{\d E}) ]}_\cong & & & H_{m-p}(D(\nu))\ar[d]^{i_*} \\
H^{p+N+M}(\R^{\NN, M}\x B, \R^{\NN, M}\x B - (\eta(E)\cup \d E)) \ar[rrr]^-{-\cap i_*[D(\nu), S(\nu) \cup D(\nu|_{\d E})]} \ar[u]^{i^*}_\cong \ar[d]_{\tau^*} & & & H_{m-p}(\R^{\NN, M} \x B) \\
H^{p+N+M}(D^{\NN, M} \x B, \d D^{\NN, M} \x B) \ar[rrr]^-{-\cap[D^{\NN, M}\x B, \d D^{\NN, M}\x B]}_\cong \ar[d]_{\Sigma \>\cong} & & & H_{m-p}(D^{\NN, M}\x B) \ar[u]_\cong \\
H^p(B)  \ar[rrr]^{ - \cap [B]}_\cong &  & & H_{m-p} (B) \ar[u]_\cong
}
\]
Commutativity of the first square at the top is the expression of the \emph{relative} Thom class as Poincar\'{e} dual to $[E,\d E]$, ie, the relative zero section.  (In fact, one can take this as the definition of the (relative) Thom class.)  
The second square commutes by the ``naturality" of the cap product, applied to 
\[
i\co (D(\nu), S(\nu) \cup D(\nu|_{\d E})) \to (\R^{\NN,M} \x B,  \R^{\NN,M} \x B - (\eta(E)\cup \d E)). 
\]
The map $i^*$ is an isomorphism by excision. 
In the third square, $D^{\NN,M}$ is a subset of $\R^{\NN,M}$ large enough to contain the image of the embedding $e$.  
The commutativity of this square comes from the ``naturality" of the cap product applied to the collapse map 
\[
\tau\co (D^{\NN,M} \x B, \d D^{\NN,M} \x B)\to(\R^{\NN,M} \x B, \R^{\NN,M} \x B - (\eta(E)\cup \d E)).
\]  
The bottom square commutes up to a sign by the relationship between cap and cross products.  By choosing the suspension isomorphism appropriately, this diagram will commute on the nose.

Finally notice that $(-\cap[B])^{-1} \circ \pi_* \circ (-\cap [E, \d E])$ is the same as going clockwise around the above diagram, from the upper-left corner to the lower-left corner.  The composition induced by the Pontrjagin--Thom collapse map corresponds to going down the left-hand side of the diagram.
Thus we have proven the lemma.
\end{proof}

\section{The triple linking number via Pontrjagin--Thom}
\label{MilnorViaPT}
In this Section, we finish the proof of Theorem 2.  We do this via neat embeddings, ie, embeddings which respect corner structure.  The reader may wish to refer back either to our brief discussion just before Corollary \ref{GluingCor} or to the references mentioned there.  We prove Theorem by first neatly embedding our glued bundle $E_g \to \L_3$ into a trivial bundle.

\subsection{Neatly embedding the glued total space}

\begin{lemma}
\label{GluedEmbedsInTrivial}
The manifold with faces $E_g$ can be embedded into a trivial bundle over $\L_3$ in such a way that it has a well defined tubular neighborhood and normal bundle.
\end{lemma}

\begin{proof}
Each of the four spaces $E_D$ (for the four diagrams $D=L,M,R,T$) can be neatly embedded 
\[
\xymatrix{ E_D \ar@{^(->}[r] & C_4[\R^3] \x \L_3}
\] 
using its definition as the pullback in square (\ref{BTsquare}) and the fact that all the corner structure in $E_D$ comes from that in $C_4[\R^3]$.  Lemma 3.2.1 of our paper \cite{Rbo} shows that the spaces $C_n[\R^d]$ can be given the structure of $\langle N \rangle$--manifolds.  So by Laures' \cite[Proposition 2.1.7]{Laures}, we can find neat embeddings over $\L_3$:
\begin{align}
\label{FirstEmbeddings}
\xymatrix{
E_T \ar@{^(->}[r]  & C_4[\R^3] \x \L_3 \ar@{^(->}[r]  & \R^M \x [0, \infty)^N \x \L_3  \\
E_L \ar@{^(->}[r]  & C_4[\R^3] \x \L_3 \ar@{^(->}[r]  & \R^M \x [0, \infty)^N \x \L_3 }
\end{align}
We reverse the orientation on the top copy of $ \R^M \x [0, \infty)^N$ and then glue the top copy of $\R^M \x [0, \infty)^N \x \L_3$ to the bottom one in such a way that the principal face of the top $C_4[\R^3]$ containing $\SS_1$ is glued via the identity map to the principal face of the bottom $C_4[\R^3]$ containing $\SS_L$.  (So far $E_T$ has not been glued to $E_L$ in any ``nice" way.)

Let $\eta(\SS_L)\cong \SS_L \x [0,1)$ denote a collar neighborhood of $\SS_L$ in $E_L$.  We have the identification $\xymatrix{g\co \SS_L \x S^2 \ar[r]^-{\cong} & \SS_1}$, and we can further identify $\eta(\SS_L) \x S^2$ with a collar $\eta(\SS_1)$ of $\SS_1$ in $E_T$.  Consider the restriction of (\ref{FirstEmbeddings}) to $\eta(\SS_1)$ and the fact that these neat embeddings are perpendicular at the boundaries of $\R^M \x [0,\infty)^N$; this allows us to smoothly extend the bottom embedding to $E_L \cup (\eta (\SS_L) \x S^2)$ by ``doubling" it near the boundary:
\begin{equation}
\label{EmbJustCollar}
\xymatrix{
E_T \cup_g (E_L \cup (\eta (\SS_L) \x S^2))  \ar@{^(->}[r] & (C_4[\R^3] \cup C_4[\R^3]) \x \L_3 \ar@{^(->}[r] & (\R^M \x [0, \infty)^N \cup \R^M \x [0, \infty)^N) \x \L_3 }
\end{equation}
Again, these are embeddings over $\L_3$.

Since the interior of the fiber $F_L$ is an open subset of Euclidean space, $F_L$ has a trivial normal bundle.  The normal bundle to $E_L \incl \R^M \x [0, \infty)^N \x \L_3$ is the normal bundle along the fiber.
Thus a tubular neighborhood of $E_L$ in $\R^M \x [0, \infty)^N \x \L_3$ can be identified with $E_L \x \R^p$ for some $p$.  (Considering the dimensions of $F_L$ and $C_4[\R^3]$, we see that $p\geq 8$.)

Consider the restriction 
\[
\xymatrix{\eta(\SS_L) \x S^2  \ar@{^(->}[r] & \R^M \x [0, \infty)^N \x \L_3 }
\]
of the embedding (\ref{EmbJustCollar}) to a collar of the face along which we glued.   We may assume that its image lies in a tubular neighborhood of $E_L$:  in fact, we can shrink the embedding of the $S^2$ factor in 
 $\eta(\SS_1) \cup_g(\eta(\SS_L) \x S^2)  \cong \SS_L \x S^2 \x (-1,1)$ along the interval $(-1,0]$ to achieve this.  Thus over each link in $\L_3$ we have an embedding 
\[
\xymatrix{F(\SS_L) \x S^2  \ar@{^(->}[r] & F(\SS_L) \x \R^p.}
\]
Let $F(\SS_L)$ denote the fiber of the face $\SS_L$ over a link in $\L_3$.  By making $p$ sufficiently large (by making $M$ above sufficiently large), we can find a smooth isotopy from the embedding above to an embedding
\[
\xymatrix{F(\SS_L) \x S^2  \ar@{^(->}[r] & F(\SS_L) \x \R^p}
\]
induced by a standard embedding $S^2 \incl \R^p$.  If we take the trace of this isotopy along part of the collar (say, along $[\frac{1}{3},\frac{2}{3}]$), we can extend outside $F(\SS_L) \x S^2 \x [0, \frac{2}{3}]$ using the standard embedding.  This gives an embedding $F_L \x S^2 \incl F_L \x \R^p$ which on the face $F(\SS_L)$ agrees with the one gotten by restricting (\ref{EmbJustCollar}).  This gives an embedding of the whole space 
\[
\xymatrix{E_L \x S^2  \ar@{^(->}[r] & E_L \x \R^p  \ar@{^(->}[r] & \R^M \x [0, \infty)^N \x \L_3 }
\]
which on $E_L \cup(\eta(\SS_L) \x S^2)$ agrees with (\ref{EmbJustCollar}).  Thus we have embedded 
\[
\xymatrix{E_T \cup_g (E_L \x S^2)  \ar@{^(->}[r] &  \left((\R^M \x [0, \infty)^N) \cup (\R^M \x [0, \infty)^N)\right) \x \L_3.} 
\]
Our construction of the embedding in (\ref{EmbJustCollar}) gives a smooth embedding of the collar of the face we glued, where the manifold-with-faces structure on the glued space is the one indicated in the proof of Lemma \ref{GluedAreMfdsWFaces}.  Furthermore, the embedding of each piece is neat and hence has a tubular neighborhood which can be identified with its normal bundle.  So the embedding above has a tubular neighborhood diffeomorphic to its normal bundle.

Continuing in a similar manner, we embed 
\begin{equation}
\label{FinalEmb}
\xymatrix{E_g  \ar@{^(->}[r] & \left(\bigcup_1^4\>  \R^M \x [0, \infty)^N\right) \x \L_3}
\end{equation}
where each of the last 3 copies of $\R^M \x [0, \infty)^N$ is glued to the first copy along a different codimension--1 face of the first copy.  The result is a subspace of $\R^M \x \R^N$.   We saw that the domain is a manifold with faces, but the codomain is not (necessarily) even a manifold with corners, since the three pairs of faces we glue along are not disjoint there.  Nonetheless, as in Lemma \ref{PrincipalFacesDisjoint}, the faces of the $E_D$ which we glue are disjoint.  Thus near any point in any $E_D$, the embedding looks the same as in the case where we glue just one pair of spaces.  Thus, as in that case, the embedding has a well defined tubular neighborhood, diffeomorphic to its normal bundle. 
\end{proof}

\subsection{Finishing the proof of Theorem 2}

Proposition \ref{Mu123IsIntegralOnGluedSpace} described the triple linking number $\mu_{123}$ as $\int_{F_g} \beta$, the integral over the fiber of a glued configuration space bundle of a certain 6--form $\beta$.  We now compare this to the Pontrjagin--Thom construction.  

Let $b_{M,N}(R)$ denote the intersection of a radius $R$ ball around the origin in $\R^{M+N}$ with $\R^M \x [0, \infty)^N$.  Let $B_{M,N}(R)$ denote the result of gluing four copies of this sector $b_{M,N}(R)$ in the same way that the four copies of $\R^M \x [0, \infty)^N$ were glued.  The embedding (\ref{FinalEmb}) gives rise to a Pontrjagin--Thom collapse map, which induces a map of pairs
\[
\xymatrix{(B_{M,N}(R) \x \L_3, \ast)  \ar[r] & (B_{M,N}(R) \x \L_3, \> B_{M,N}(R) \x \L_3 - \eta(E_g) )}
\]
for sufficiently large $R$, where $\eta(E_g)$ is a tubular neighborhood of $E_g$ and where the basepoint $\ast$ is a point of distance $R$ from the origin in $B_{M,N}(R)$.  The collapse map further descends to the map of pairs below
\[
\tau\co (B_{M,N}(R) \x \L_3,\, \d(B_{M,N}(R)) \x \L_3) \xymatrix{ \ar[r] & } (B_{M,N}(R) \x \L_3, \, (B_{M,N}(R) \x \L_3 - \eta(E_g)) \, \cup \, \d E_g )
\]
which on the level of quotients is
\[
\xymatrix{\tau\co \Sigma^{M+N} \L_3 \ar[r] & E_g^{\nu_{M,N}} / \d E_g^{\nu_{M,N}}}
\]
where $E_g^{\nu_{M,N}}$ is the Thom space of the normal bundle $\nu_{M,N} \to E_g$ of our embedding of $E_g$ in a trivial bundle.

We have already seen that, by our choices of spherical forms and vanishing arguments for infinite faces, the form $\beta$ represents a class $[\beta] \in H^6(E_g, \d E_g)$.  We claim that by Lemma \ref{IntegrationIsPTBoundary}, the image $\tau^*[\beta]$ 
of $[\beta]$ under the map in cohomology induced by the collapse map above 
is precisely $\int_{F_g} \beta = \mu_{123}$.  Strictly, Lemma \ref{IntegrationIsPTBoundary} only shows this when the base is a finite-dimensional compact manifold.  However, we can consider the smooth structure on $\L_3$ as being defined by declaring (in an obvious way) what the smooth maps from finite-dimensional manifolds are.  Then Lemma \ref{IntegrationIsPTBoundary} applies to any compact submanifold of $\L_3$.  
Since $\tau^*\beta$ and $\int_{F_g} \beta$ agree over any compact submanifold, they are indeed equal.  
This proves Theorem 2.
\qed

    \bibliographystyle{alpha}
    \bibliography{refs}

\end{document}